\newcommand{\R}{\mathbb{R}}
\newcommand{\T}{\mathbb{T}}
\newcommand{\NN}{\mathbb{N}}
\newcommand{\ZZ}{\mathbb{Z}}
\newcommand{\nn}{\mathrm{n}}

\newcommand{\bd}{\partial}
\newcommand{\mix}{\mathrm{mix}}
\newcommand{\Til}{\mathcal{T}}
\newcommand{\Qone}{\mathscr{Q}}
\newcommand{\dive}{\mathop{\mathrm{div}}}

\newcommand{\ove}[1]{\smash{\overline{#1}}}
\newcommand{\longtilde}[1]{\smash{\widetilde{#1}}}
\def\mint{\mkern12mu\hbox{\vrule height4pt depth-3.2pt width5pt}\mkern-16.5mu\int}

\documentclass[11pt,reqno,a4paper]{amsart}

\usepackage{mathrsfs}
\usepackage{graphicx}
\usepackage{amssymb}
\usepackage{fancyhdr}
\usepackage{hyperref}\hypersetup{colorlinks=true, citecolor=blue}
\usepackage{enumitem}
\numberwithin{equation}{section}

\newtheoremstyle{mytheorem}
{}% measure of space to leave above the theorem. E.g.: 3pt
{}% measure of space to leave below the theorem. E.g.: 3pt
{\it}% name of font to use in the body of the theorem
{\parindent}% measure of space to indent
{\bf}% name of head font
{.}% punctuation between head and body
{ }% space after theorem head; " " = normal interword space
{\thmnumber{#2.~}\thmname{#1}\thmnote{~\rm#3}}% Manually specify head

\newtheoremstyle{myremark}
{}% measure of space to leave above the theorem. E.g.: 3pt
{}% measure of space to leave below the theorem. E.g.: 3pt
{\rm}% name of font to use in the body of the theorem
{\parindent}% measure of space to indent
{\bf}% name of head font
{.}% punctuation between head and body
{ }% space after theorem head; " " = normal interword space
{\thmnumber{#2.~}\thmname{#1}\thmnote{~\rm#3}}% Manually specify head

\newtheoremstyle{myparagraph}
{}% measure of space to leave above the theorem. E.g.: 3pt
{}% measure of space to leave below the theorem. E.g.: 3pt
{\rm}% name of font to use in the body of the theorem
{\parindent}% measure of space to indent
{\bf}% name of head font
{.}% punctuation between head and body
{ }% space after theorem head; " " = normal interword space
{\thmnumber{#2.~}\thmname{#1}\thmnote{#3}}% Manually specify head

\theoremstyle{mytheorem}
\newtheorem{theorem}[subsection]{Theorem}
\newtheorem{lemma}[subsection]{Lemma}

\newtheorem{proposition}[subsection]{Proposition}

\theoremstyle{myremark}
\newtheorem{remark}[subsection]{Remark}

\newtheorem{definition}[subsection]{Definition}

\theoremstyle{myparagraph}
\newtheorem{parag}[subsection]{}
\newtheorem*{parag*}{}

\makeatletter
\def\@secnumfont{\sc}
\def\section{\@startsection%
{section}	% name: section/subsection/etc.
{1}			% level: 1 for section/2 for subsection/etc.
\z@{1.5\linespacing\@plus .2\linespacing}	% vertical skip before
  {.7\linespacing}	% vertical skip after
  {\normalfont\sc\centering}}	% style
\makeatother
\makeatletter
\renewenvironment{proof}[1][\proofname]{\par
  \pushQED{\qed}%
  \normalfont \topsep6\p@\@plus6\p@\relax
  \trivlist
  \itemindent\normalparindent
  \item[\hskip\labelsep
    \bfseries
    #1\@addpunct{.}]\ignorespaces
}{%
  \popQED\endtrivlist\@endpefalse
}
\providecommand{\proofname}{Proof}
\makeatother

\newcommand{\footnoteb}[1]{\footnote{~#1}}

\begin{document}

\title[Exponential self-similar mixing]{ Exponential self-similar mixing by incompressible flows}

\author[G. Alberti]{Giovanni Alberti}
\address{Dipartimento di Matematica,
Universit\`a di Pisa, largo Pontecorvo 5, I-56127 Pisa, Italy}
\email{giovanni.alberti@unipi.it}

\author[G. Crippa]{Gianluca Crippa}
\address{Departement Mathematik und Informatik,
Universit\"at Basel, Spiegelgasse 1, CH-4051 Basel,
Switzerland}
\email{gianluca.crippa@unibas.ch}

\author[A. L. Mazzucato]{Anna L. Mazzucato}
\address{Department of Mathematics, Penn State University, University Park, PA, 16802, U.S.A.}
\email{alm24@psu.edu}

\dedicatory{Dedicated to Alberto Bressan and Charles R.~Doering\\
on the occasion of their 60\,$^{th}$ birthday}

\date{\today}

\keywords{mixing, continuity equation, negative Sobolev norms, incompressible 
flows, self-similarity, potentials, regular Lagrangian flows.}

\subjclass[2010]{35Q35, 76F25}

\begin{abstract}
We study the problem of the optimal mixing of a passive scalar under 
the action of an incompressible flow in two space dimensions. 
The scalar solves the continuity equation with a divergence-free 
velocity field, which satisfies a bound in the Sobolev space $W^{s,p}$, 
where $s \geq 0$ and $1\leq p\leq \infty$.
The mixing properties are given in terms of a characteristic length
scale, called the mixing scale.
We consider two notions of mixing scale, one functional, expressed
in terms of the  homogeneous Sobolev norm~$\dot H^{-1}$, the other
geometric, related to rearrangements of sets.
We study rates of decay in time of both scales under self-similar
mixing. 
For the case~$s=1$ and $1 \leq p \leq \infty$ (including the case 
of Lipschitz continuous velocities, and the case of physical interest
of enstrophy-constrained flows), we present examples of velocity
fields and initial configurations for the scalar that saturate
the exponential lower bound, established in previous works, on the
time decay of both scales.
We also present several consequences for the geometry of regular
Lagrangian flows associated to Sobolev velocity fields.
\end{abstract}
\maketitle

\tableofcontents

%
%	INTRODUCTION
%

\section{Introduction}
\label{intro}
We study the problem of optimal mixing of scalar, passive tracers 
by incompressible flows.
How well a quantity transported by a flow is mixed is an important 
problem in fluid mechanics and in many applied fields, for instance
in atmospheric and oceanographic science, in biology, and in chemistry. 
In combustion, for example, fuel and air need to be well mixed for an
efficient reaction to take place. 
In many situations, the interaction between the tracer and the flow
can be neglected: mathematically, this results in the fact that the
tracer solves a linear continuity equation with a given velocity 
field (see~\eqref{e:continuity}).
This problem is also a surprisingly rich source of questions in
analysis, in particular relating partial differential equations
and dynamical systems with geometric measure theory.

\medskip

There is a well-established fluid mechanics literature concerning 
mixing and turbulence, especially with respect to statistical
properties (see e.g.~\cite{BCCLV00,GW12} and references therein).
It is known, in fact, that turbulent advection enhances mixing,
which in turn can enhance diffusion and suppress concentration
(see \cite{CKRZ} for steady ``relaxation enhancing" flows and 
\cite{KXchemo} for an application to chemotaxis, for instance). 
Enhanced dissipation occurs also in Euler flows as an effect 
of inviscid Landau damping (see \cite{BMV} and references therein).
Mixing has also long been studied in the context of chaotic 
dynamics \cite{Aref,Ottino,liverani}. 
Indeed the decay to zero of the mixing scale defined in terms of
negative Sobolev norms corresponds to ergodic mixing by the flow 
(as shown in~\cite{MMGVP}), and several well-known examples of 
discrete dynamical systems exhibit an exponential decay of
correlations, which essentially means exponential mixing (however, 
these examples cannot be easily adapted to our context).

Recently there has been a renewed interest in quantifying the 
degree of mixing under an incompressible flow, and in producing 
examples that achieve optimal mixing. On the analytic side, 
progress has been possible in part due to the development of 
new tools to study transport and continuity equations under 
non-Lipschitz velocities~\cite{DPL,Amb, HW}, in particular
quantitative estimates on regular Lagrangian flows \cite{CDL}. 
On the applied and computational side, optimal mixing has been 
approached from the point of view of homogenization and control 
with more realistic models~\cite{Liu08,FCS14}. 
Experiments have also been performed (see for 
example~\cite{GDTR09,JCT00,Jul03}).

\begin{parag}[The continuity equation]
\label{ss:cont}
We consider mixing in two space dimensions, as 2D is the first 
dimension with non-trivial, divergence-free fields and also for 
comparison with computational and experimental studies. 
Generally, dimension will not play a crucial role in what follows, 
except in setting scaling laws.
However, it is technically more difficult to construct optimal 
mixers in two space dimensions, informally speaking for topological
reasons. In fact, all our results can be extended to higher 
dimensions in a straightforward manner by making all quantities 
constant with respect to the additional independent variables. 
The divergence-free condition is a strong constraint that can be 
somewhat relaxed, but it is physically motivated in applications 
of mixing, and it is essential for the definition of the mixing 
scales we adopt.
In fact, since we aim at producing examples of optimal mixing, 
the divergence-free condition is a more restrictive requirement 
that must be satisfied in our constructions.

\medskip
We work on the two-dimensional torus $\T^2 := \R^2 / \ZZ^2$ 
or on the plane $\R^2$.
When considering the plane $\R^2$, both velocity fields and 
solutions eventually resulting from our constructions will 
be supported in a fixed compact set.

Given a divergence-free, time-dependent velocity field $u = u(t,x)$,
we consider a  scalar $\rho = \rho(t,x) \in L^\infty$ that is passively 
advected by $u$, i.e., a solution of the transport equation:
\begin{equation}\label{e:transport}
  \bd_t \rho + u \cdot \nabla \rho = 0.
\end{equation}
Under the divergence-free assumption on the velocity $u$, 
the scalar $\rho$ is also a solution of the continuity equation:
\begin{equation}
\label{e:continuity}
  \bd_t \rho + \dive (u \rho) = 0 \,.
\end{equation}
We  prescribe an initial datum $\rho(0,\cdot) = \bar \rho$ at time $t=0$.

In the following we will always assume 
that the initial datum $\bar \rho$ has integral equal to $0$.
Since the continuity equation~\eqref{e:continuity} preserves the integral 
of the solution over the spatial domain along the time evolution, 
it follows that $\rho(t,\cdot)$ has zero integral for any time $t$.
This fact is relevant when using negative Sobolev norms to measure mixing 
(see Definition~\ref{def:Functscale}, \S\ref{ss:defH}, and 
Remarks~\ref{r:fctsinH}\ref{r:fctsinH:1},~\ref{r:fctsinH}\ref{r:fctsinH:2}). 

\end{parag}

\begin{parag}[Functional and geometric mixing scales]
In order to discuss the mixing properties of solutions to the continuity
equation~\eqref{e:continuity} we need to define a notion of mixing scale 
that can quantify the ``level of mixedness'' of the solution $\rho(t,\cdot)$
at time $t$. 
At least at a formal level, the continuity equation preserves  all $L^p$
norms of the solution, which as a result are not a suitable measurement of
mixing in our setting.%
\footnoteb{In fact, $L^p$ norms of the solutions are frequently used as a
measurement of the mixing scale for solutions of advection-diffusion
equations, i.e., in the case when $\rho$ solves $\bd_t \rho + \dive(u\rho)
= \Delta \rho$. Due to the viscosity, $L^p$ norms of the solution are 
dissipated along the time evolution.}
Though, it is still possible for $\rho(t,\cdot)$ to converge to zero weakly.%
\footnoteb{Using characteristic functions of sets as test functions,
it is not difficult to prove that this will be the case for instance
if the flow of $u$ is strongly mixing in the ergodic sense.}
This is the mixing process we want to quantify and analyze in this paper. 
\medskip

We will employ and compare two notions of mixing scales that are
considered in the literature. The first one is based on a negative
Sobolev norm of the solution~$\rho(t,\cdot)$, more precisely the norm 
in the homogeneous Sobolev space  $\dot{H}^{-1}$ following~\cite{doering}
(see Definition~\ref{def:Functscale} below, and see~\S\ref{ss:defH} 
for the definition of homogeneous Sobolev norms), and will be referred 
to as the {\em functional mixing scale}.%
\footnoteb{From a mathematical point of view there is nothing special 
with the order $-1$ that has been chosen in the definition of functional 
mixing scale: every negative Sobolev norm would behave in a similar way. 
However, from a physical point of view, this choice is the most convenient,
since the norm in $\dot{H}^{-1}$ scales as a length on the two-dimensional
torus.}
In fact, it can be proven that the vanishing of the homogeneous
negative Sobolev~$\dot H^{-1}$ norm of $\rho(t,\cdot)$ is equivalent to 
the convergence of $\rho(t,\cdot)$ to $0$ weakly in $L^2$ (see for 
instance~\cite{doering}). 
The use of negative norms to measure mixing was proposed in~\cite{MMP}, 
where the equivalence between the decay of the $\dot H^{-1/2}$
norm and mixing in the ergodic sense was established.
The second mixing scale arises from a conjecture of Bressan~\cite{bressan}
on the cost of rearrangements of sets  and brings in a connection with 
geometric measure theory.
This second notion of scale is expresses in terms of how small the mean 
of the solution $\rho(t,\cdot)$  is on suitably small balls (see 
Definition~\ref{d:geomix} below), and it will be referred to as the
{\em geometric mixing scale}. 
The two scales are related though generally not equivalent.

\medskip
\emph{In the rest of this introduction, we informally denote
any of the two mixing scales of the solution $\rho$ at time $t$ by
$\mix\big( \rho(t,\cdot) \big)$}.

\medskip
Ideally, a flow that ``mixes optimally'' will achieve the largest
decay rate in time for $\mix\big( \rho(t,\cdot) \big)$.
How fast $\mix\big( \rho(t,\cdot) \big)$ can decay in time depends
on properties of the flow.
These, in turn, are in practice given in terms of constraints
on certain quantities of physical interest, typically energy, 
enstrophy, and palenstrophy.
These correspond respectively to uniform-in-time bounds on the 
$L^2$, $H^1$, and $H^2$ norms of the velocity field $u$.
\end{parag}

\begin{parag}[Main results]
\label{ss:intromain}
As described in detail in \S\ref{ss:literature} below, it has been
recently proven~\cite{CDL,ikx,seis} that the (functional or geometric) 
mixing scale can decay at most exponentially in time:
\begin{equation}
\label{e:introbound}
\mix\big( \rho(t,\cdot) \big) \geq C \exp ( -ct) \,,
\end{equation}
if the velocity field satisfies a constraint on the Sobolev 
norm $W^{1,p}$  for some $1 < p \leq \infty$, uniformly in time.
Above, $C>0$ and $c>0$ are constant depending on the initial datum
$\bar \rho$ and on the given bounds on the velocity field. 

The primary goal of this work is to show the optimality of the bound
\eqref{e:introbound} for all $1<p\leq \infty$, which  was previously
unknown (see however~\cite{zlatos} and the brief description 
in~\S\ref{ss:literature} below).
Our strongest result concerning the decay of the mixing scale can be 
stated as follows:

\smallskip
\emph{There exist a smooth, bounded, divergence-free velocity field
$u$ which is Lipschitz uniformly in time, and a smooth, bounded, nontrivial
solution $\rho$ of the continuity equation~\eqref{e:continuity} such that
the (functional or geometric) mixing scale of the solution decays exponentially
in time:
\begin{equation}\label{e:intromain}
\mix\big( \rho(t,\cdot) \big) \leq C \exp (-ct).
\end{equation}
%
%which shows that the lower bound~\eqref{e:introbound} is optimal.
}
\end{parag}

\begin{remark}\label{r:remmain}
\quad
%%%%%
\begin{enumerate}
[label=(\roman*),ref=(\roman*), 
leftmargin=0pt, itemsep=2pt, itemindent=30pt]

\item 
It is very important to keep in mind the difference between the regularity
allowed on the velocity field itself and  the regularity spaces where the
velocity satisfies bounds uniformly in time. 
In the above statement, describing our strongest result, the velocity field
is smooth in space and time.
However, the velocity is uniformly bounded in time only in the Sobolev spaces
$W^{s,p}$ with $0 \leq s \leq 1$ and $1 \leq p \leq \infty$.
If $s>1$, the velocity is bounded in $W^{s,p}$ on any finite time interval
$[0,T]$, $0<T<\infty$, 
but  the norm blows up when~$t \to \infty$.

\item 
\label{r:remmain:ii} 
The construction that leads to the main result stated above also yields 
examples of regular velocity fields and smooth solutions exhibiting different 
rates of decay for the mixing scale which depend on the uniform-in-time 
bounds for the Sobolev norms of the velocity field. 
More precisely, if we  ask that the velocity field is uniformly 
bounded in time  in the Sobolev space $W^{s,p}$ for some $s\geq 0$ 
and some $1\leq p\leq \infty$, then we can construct examples such that:
%%%%%
\begin{itemize}
[itemsep=2pt]
\item 
If $s<1$, there exists a time $t^*$ such that 
$\mix\big(\rho(t^*,\cdot)\big)=0$, that is, 
perfect mixing is achieved in finite time;
\item
If $s=1$, the mixing scale decays exponentially, that is, \eqref{e:intromain} holds;
\item
If $s>1$, the mixing scale decays polynomially, that is, there exists an exponent
$\alpha=\alpha(s)>0$ such that $\mix\big( \rho(t,\cdot) \big) \leq C t^{-\alpha}$.
\end{itemize}
\end{enumerate}
\end{remark}

Further remarks are detailed in \S\ref{ss:introself} 
and \S\ref{ss:introrem} below. The results presented in this article 
were announced in \cite{ACM}.

Before making further observations on our results and techniques we make a
digression about the past literature on this topic.

\begin{parag}[Past literature]
\label{ss:literature}
Mixing phenomena are studied in the literature under energetic
constraints on the velocity field, that is, assuming that the
velocity field is bounded with respect to some spatial norm, 
uniformly in time. 
This research area is related in a very natural way to the study
of transport and continuity equations under non-Lipschitz
velocities (see~\cite{HW} for a recent survey).
We survey key results in the literature on both areas
(most of the results hold in any space dimension):
%%%%%
\begin{enumerate}
[label=(\alph*), ref=(\alph*),
leftmargin=0pt, itemindent=30pt, itemsep=3pt]
%leftmargin=30pt, itemsep=2pt]
%\itemsep 2 pt\leftskip -10 pt
\item
The velocity field $u$ is bounded in $W^{s,p}$ uniformly in time for
some $s<1$ and $1\leq p\leq \infty$ (the case $s=0$, $p=2$, relevant 
for applications, is often referred to as energy-constrained flow).
In this case, in general there is  no uniqueness for the solution 
to the Cauchy problem for the continuity 
equation~\eqref{e:continuity} (see~\cite{ABC2,ABC1}). 
Hence, one can find a velocity field and a bounded solution which
is non-zero at the initial time, but is identically zero at some
later time.
Therefore it is possible to have perfect mixing in finite time,
as already observed in \cite{doering} and established in~\cite{llnmd}
for $s=0$, building on examples from~\cite{depauw,bressan}.

\item
The velocity field $u$ is bounded in $W^{1,p}$ uniformly in time
for some $1\leq p\leq \infty$ (the case $p=2$, relevant for 
applications, is often referred to as enstrophy-constrained flow).
The theory in \cite{DPL} guarantees uniqueness for the Cauchy 
problem~\eqref{e:continuity}, which in particular excludes perfect 
mixing in finite time. 
A quantification of the maximal decay rate for the mixing scale 
has been achieved thanks to the quantitative estimates for regular
Lagrangian flows in~\cite{CDL}.
In detail, for $p>1$, the theory in~\cite{CDL} provides an exponential
lower bound on the geometric mixing scale (see~\eqref{e:introbound}).
The extension to the borderline case $p=1$ is still open
(see, however, \cite{BC}).
The same exponential lower bound~\eqref{e:introbound}
has been proved for the functional mixing scale in~\cite{ikx,seis}.
See also ~\cite{leger, sharploss} for further results on these bounds. 
More recently, in~\cite{bresch} the authors were able to prove regularity
estimates for the solution of the continuity equation by studying the
propagation of a weighted norm of the solution, without the assumption
of bounded divergence on the velocity.

\item
The theory in~\cite{Amb} provides  uniqueness for the Cauchy 
problem~\eqref{e:continuity} for velocity fields bounded in $BV$
uniformly in time (see also~\cite{bonicatto}, in which the
divergence-free assumption is replaced by the more general
condition of near incompressibility).
Again, uniqueness  excludes perfect mixing in finite time.
The validity of the bound~\eqref{e:introbound} is still
unknown in this context.
However, in~\cite{bressan} it is observed that such an exponential
decay of the geometric mixing scale can indeed be attained for
velocity fields bounded in $BV$ uniformly in time.
The same example works also for the functional mixing scale.

\item \label{ss:literature:d}
The velocity field $u$ is bounded in $W^{s,p}$ uniformly in
time for some $s>1$ and $1\leq p\leq \infty$ (the case $s=2$
and $p=2$, relevant for applications, goes under the name of
palenstrophy-constrained flow).
In this case, Estimate~\eqref{e:introbound} gives immediately
an exponential lower bound for both mixing scales.
However, it is still open whether such a bound is sharp or not.
Numerical simulations, such as those in~\cite{llnmd,ikx},
and heuristic arguments support the optimality of the exponential decay.
(See also the discussion in~\S\ref{ss:introrem}.) 
\end{enumerate}

The  constants $C$ and $c$ in~\eqref{e:introbound} depend 
not only on the given bounds for the velocity field, but
also on the initial datum $\bar\rho$ (not simply through 
its mixing scale). 
In fact, it is not clear that an estimate of the form
\[
\mix\big(  \rho(t,\cdot) \big) 
\geq C \, \mix( \bar\rho  ) \exp (-ct),
\]
with $C$ and $c$ constants depending only on the given
bounds on the velocity field, can be achieved.
It is then natural to ask whether it is possible 
to obtain bounds on the rate of decay with constants 
that only depend on the mixing scale of the initial datum, 
and not on its geometry. Unfortunately, direct PDE methods, 
such as energy estimate, do not seem to yield sharp bounds: 
for instance, they yield a Gaussian bound for 
palenstrophy-constrained flows, while the optimal
bound is at least exponential (see~\cite{llnmd}).

\medskip

There are  examples in the literature of enstrophy-constrained 
flows that saturate the exponential decay rate complementary 
to those presented in this work~(\S\ref{ss:intromain}). 
Yao and Zlato\v s~\cite{zlatos} utilize a cellular flow to 
obtain decay of the mixing scale for any  bounded initial 
datum $\Bar{\rho}$  (where the flow depends on $\Bar\rho$), 
under a $W^{1,p}$ constraint on the velocity field for any 
$1 \leq p\leq\infty$. The decay rate is optimal in the range 
$1 < p < \bar p$ for some explicit $\bar p >2$. 
They also give an interesting result on ``unmixing" a given 
configuration.
(See also \S\ref{ss:introrem} for a comparison of our 
results with those from~\cite{zlatos}.)

\medskip
Before these recent analytic results, numerical experiments 
were performed that supported an exponential  rate of decay
for $\mix\big( \rho(t,\cdot) \big)$ under an enstrophy 
constraint.
For instance, a numerical scheme to compute an instantaneous
optimizer was given in~\cite{doering}, and numerical tests
performed for a sinusoidal initial configuration.
A global optimizer was computed numerically in~\cite{MMGVP}.

\medskip
Our approach to finding optimal mixers is constructive and
is essentially based on a self-similar scheme.
We actually present two related, but distinct constructions: 
the first one, referred to as the \emph{self-similar construction}, 
is simpler and self-similar in a strict sense.
This first construction allows us to obtain only examples where the
velocity field is neither smooth nor uniformly bounded in $W^{1,\infty}$.
The second construction, which we refer to as the 
\emph{quasi-self-similar construction}, is more involved and allows
us to construct examples where the velocity is smooth, and uniformly
bounded in $W^{1,\infty}$.

We do not claim that a self-similar evolution is more physical 
(however, see~\cite{RHG99}) or preferred over other types.
The only reason for choosing (quasi) self-similar constructions
is that it makes the mixing scale of (some) solutions easier
to estimate.
\end{parag}

\begin{parag}[Self-similar construction]
\label{ss:introself}

Briefly, the construction starts from a ``basic move", which 
is just a pair of a velocity field and a weak solution of the 
associated continuity equation (quite often the characteristic 
function of a set) and consists in combining  infinitely many 
copies of this basic move, suitably rescaled in time and space, 
so to obtain a velocity field and a solution with the desired 
features (the reader can glimpse at the self-similar construction 
in Figure~\ref{f:self}).
More precisely, the construction is divided in three steps:
%%%%%
\begin{enumerate}
[label=Step~\Roman*., ref=\Roman*,
itemsep=2pt, leftmargin=50pt]

\item\label{step1} 
Scaling analysis (Section~\ref{selfsimilar}): 
we assume the existence of a basic move (velocity field and solution) 
defined for the times $0 \le t \le 1$  with a certain regularity, 
and describe the self-similar construction that gives a new velocity 
field and solution defined for all times $t\ge 0$. 
We then analyze the decay of the mixing scale of this new solution 
and the behavior of the Sobolev norms of the new velocity field.

\item\label{step2} 
Geometric tools (Section~\ref{basic}): 
we establish a series of geometric lemmas that guarantee 
the existence of smooth, divergence-free velocity fields with 
the property that the associated flows deform smooth sets
according to a prescribed evolution in time.

\item \label{step3} 
Construction of the basic move  (Section~\ref{pinching}):   
we use such geometric tools to construct the basic move we need in 
Step~\ref{step1}. 
The solution is the characteristic function of a regular set that evolves 
smoothly for all times except finitely many singular times.
The main technical point here is to deal with the possible singularities present.
\end{enumerate}

We stress that the regularity of (and the bounds on) the velocity 
field and the regularity of the solution constructed in Step~\ref{step1} 
above depend only on the regularity of (and the bounds on) the basic move. 
So far, using  a strictly self-similar approach, we have only been able
to construct a basic move with velocity of class $W^{1,p}$ with $p<\infty$.
\end{parag}

\begin{parag}[Quasi-self-similar construction]
To construct examples of exponential decay
of the mixing scale of the solution with a velocity field 
which is bounded in $W^{1,\infty}$ uniformly in time
(as claimed in~\S\ref{ss:intromain}), we use a construction which is 
not exactly self-similar. 
The main difference is that we combine rescaled copies not of 
just one basic move, but of a (finite) family of basic moves.
Using this more flexible setting we can actually 
construct velocity fields and solutions that are
smooth in both space and time.
	
Of the three steps mentioned above, Step~\ref{step1} is 
unchanged except that we assume the existence of a family
of (smooth) basic moves (see Section~\ref{quasiselfsimilar}). 
Step~\ref{step2} consists of some improvements of the 
geometric lemmas established in Section~\ref{basic}
(see Section~\ref{geo2}).
Finally Step~\ref{step3} consists as before of the 
construction of the basic moves (Section~\ref{snake}). 
In this second construction, the regularity of the velocity
field does not depend just on the regularity of the basic moves,
but also on the details of how the rescaled moves are glued
together and on the overall combinatorial aspects of the
construction (see for instance Figure~\ref{f:gamma-i+f}).  

Given that the basic moves are smooth and hence in $W^{s,p}$ 
for any $s\geq 0$ and $1\leq p\leq \infty$, the analysis in 
Step~\ref{step1} can be performed for any such $s$ and $p$. 
However, the bounds on the Sobolev norms of the velocity 
field and the decay of the mixing scale depend on the specific
rescaling of the basic moves.

It turns out that an exponential decay of the
mixing scale of the solution can only be obtained with 
a uniform bound on the Sobolev norms of the velocity with $s=1$
(regardless of $p$). 
Vice versa, examples with a uniform bound on the Sobolev norms 
of the velocity with $s>1$ can only be obtained with a polynomial
decay of the mixing scale of the solution
(see the discussion in Remark~\ref{r:remmain}\ref{r:remmain:ii}). 
In particular, in our examples of palenstrophy-constrained velocity 
the mixing scale decays only at a polynomial rate
rather than the expected exponential rate
(recall~\S\ref{ss:literature}\ref{ss:literature:d}). 
Whether exponential rate can be obtained is open in this case.
\end{parag}

\begin{parag}[Further remarks and open problems]
\label{ss:introrem}
Our main result is a proof that the bound~\eqref{e:introbound} is optimal.
In order to do so, we construct one velocity field which is bounded in 
$W^{1,p}$ uniformly in time for any~$1 \leq p \leq \infty$,
and one solution, the mixing scale of which decays exponentially. 
In fact, following the strategy described above, it is possible to
construct a large class of initial data for which the mixing scale
decays exponentially.
However, it is unclear whether this is the case for every initial datum.
The following questions about the existence of ``universal mixers''
are therefore  natural.
%%%%%
\begin{enumerate}
[label=(\alph*),ref=(\alph*),
itemsep=2pt, leftmargin=30 pt]

\item\label{q.1} 
Given any bounded initial datum $\bar \rho$, is there a velocity
field, bounded in~$W^{1,p}$ uniformly in time and possibly 
dependent on  $\bar \rho$, such that $\mix\big( \rho(t,\cdot) \big)$
decays to zero? 

\item\label{q.2} 
If the answer to Question~\ref{q.1}  is positive, does this velocity field
drive the mixing scale of $\rho$ to zero exponentially in time? 
Since, in principle, answers to Question~\ref{q.1}  may not be unique
given $\Bar\rho$, we are seeking at least one such velocity field. 

\item\label{q.3}
Does there  exist one velocity field, bounded in~$W^{1,p}$ uniformly in time,
such that $\mix\big( \rho(t,\cdot) \big)$ decays to zero for every bounded
initial datum $\bar \rho$? 

\item \label{q.4}
If the answer to Question~\ref{q.3} is positive, does this velocity field
drive the mixing scale of $\rho$ to zero exponentially in time? 
The same comment about uniqueness in Question~\ref{q.2} applies here.  
\end{enumerate}

We observed in Remark~\ref{r:remmain}\ref{r:remmain:ii} 
that our construction provides an example of palenstrophy-constrained 
flow such that $\mix\big( \rho(t,\cdot) \big)$
decays polynomially in time. In fact, the self-similarity ansatz 
implies polynomial decay of $\mix\big( \rho(t,\cdot) \big)$
under the assumption that the velocity field is bounded in $W^{s,p}$ 
uniformly in time for some $s>1$. 
However, the numerical results mentioned 
in~\S\ref{ss:literature}\ref{ss:literature:d} 
support an exponential decay also for $s>1$, but the optimal 
bound is still unknown in this case.
If the optimal decay were indeed exponential, we would then 
deduce that for $s>1$ self-similarity is too restrictive and only 
allows for sub-optimal decay rates. 
Such a result  would be in stark contrast with the case $s=1$,
for which the optimal decay rate can be achieved with  
a self-similar evolution.
We therefore formulate the following important
question:
%%%%%
\begin{enumerate}
[start=5,label=(\alph*),ref=(\alph*),
itemsep=2pt, leftmargin=30 pt]
\item 
Do a bounded initial datum $\bar \rho$ and a velocity field, 
which is bounded in $W^{s,p}$ uniformly in time for some $s>1$, 
exist such that $\mix\big( \rho(t,\cdot) \big)$ 
decays exponentially in time?
\end{enumerate}
Such an example could not then be self-similar. 
In addition, the analysis in~\cite{schulze} implies that it 
cannot be realized with a ``localized'' flow: roughly speaking,
once the solution has been mixed to a certain scale, it can
be more convenient to let the flow act again at larger scales 
before reaching a lower mixing scale.

\medskip

As mentioned in~\S\ref{ss:literature}, examples of 
enstrophy-constrained flows that saturate the exponential decay
rate~\eqref{e:introbound} have been constructed in~\cite{zlatos}.
There, the authors utilize a cellular flow consisting of 
pseudo-rotations on a family of nested tilings of the square,
and are able to obtain exponential mixing of every bounded
initial datum $\Bar\rho$ by means of a velocity field, which
depends on $\Bar\rho$ in general,  bounded in $W^{1,p}$ uniformly
in time in the range $1 \leq p< \bar p$ for some $\bar p > 2$. 
Therefore, this construction provides a partial answer to
Question~\ref{q.2} above.
Their nice geometric argument is based on a ``stopping time'' 
for the pseudo-rotation, which is determined by a clever
application of the intermediate value theorem for continuous
functions.
Their construction also applies in the range 
$\bar p \leq p \leq \infty$, giving a mixing rate which is 
slightly slower than exponential, thus answering Question~\ref{q.1}
above.

In comparison with~\cite{zlatos}, while we obtain exponential 
mixing of the solution in the full range $1 \leq p \leq \infty$,
including the Lipschitz case, our construction applies only to 
certain specific initial data.
Our strategy has a geometric flavor and generates velocity fields
and solutions that are smooth.
This last fact is relevant for the full scaling analysis (recall 
Remark~\ref{r:remmain}\ref{r:remmain:ii}) and for the application to
the study of the loss of regularity for continuity equations, which 
is addressed in the companion paper~\cite{loss} (see~\S\ref{ss:loss} 
for a brief discussion). 
In addition, our examples provide an important insight
into the geometrical properties of regular Lagrangian flows.
\end{parag}

\begin{parag}[Geometry of regular Lagrangian flows]
\label{ss:geomrlf} 
When the velocity field is Lipschitz (as in the quasi-self-similar 
examples) then the associated flow is well-defined in the classical
sense, and there is not much to add.
However, when the velocity field has singularities and belongs
only to some Sobolev class (as in the self-similar examples) 
then the flow is no longer well defined in the classical sense, 
and one should instead consider the notion of regular Lagrangian flow.%
\footnoteb{The notion of  regular Lagrangian flow is the appropriate 
one for the flow generated by an ordinary differential equation
(ODE for short) for which the velocity field has low regularity. 
A regular Lagrangian flow in $\R^d$ solves the ODE for almost every
initial point, and additionally preserves the $d$-dimensional Lebesgue
measure up to a bounded factor. 
The theory in~\cite{DPL,Amb} guarantees that, if the velocity field
is Sobolev or $BV$ and has bounded divergence, then there exists
a unique regular Lagrangian flow associated to it. Regular Lagrangian flows associated
to vector fields with zero 
divergence, as in the examples we construct, are volume-preserving.}
We stress that the regular Lagrangian flow associated to our
self-similar examples has the following additional properties:
%%%%%
\begin{itemize}
[leftmargin=30pt, itemsep=2pt]

\item 
the associated regular Lagrangian flow does not preserve the 
property of a set of being connected;
\item
there exists a segment that is collapsed to a point and, subsequently, 
inflated back to a full segment in finite time under this regular 
Lagrangian flow;%
\footnoteb{By definition, a regular Lagrangian flow in $\R^d$
does not compress $d$-dimensional sets to null set. We see here
that it can compress $1$-dimensional sets to $0$-dimensional sets.}
\item 
as a consequence, the trajectories of the velocity field
(that is, the solutions of the associated ODE) which start 
at a point in this segment are non unique.
\end{itemize}
\end{parag}

\begin{parag}[Loss of regularity for continuity equations]
\label{ss:loss}
Mixing leads to growth of positive Sobolev norms of the solution
$\rho$, saturating the exponential growth which follows from the
classical Gr\"onwall inequality. Analytically, this result is a 
consequence of the preservation of the $L^2$-norm of the solution 
and of the exponential decay of the negative Sobolev norms
in~\eqref{e:intromain} by an interpolation argument.

In the companion paper~\cite{loss},  we  present an example of a 
velocity field in~$W^{1,p}$ for any $1 \leq p < \infty$ that is 
regular except at a point and of a smooth $\bar\rho$, such that the
corresponding solution of~\eqref{e:continuity} leaves any Sobolev
space~$H^s$ with~$s>0$ instantaneously for~$t>0$.
Extensions of this construction to non-Lipschitz fields with Sobolev
regularity of order higher than $1$ are also possible.

Lack of propagation of $C^0$ and of $BV$ regularity for solutions 
of the continuity equation was already observed in~\cite{ferruccio}.
More recently, in~\cite{Jab15} it was observed that Sobolev
regularity of order one does not transfer from a velocity field
to its associated flow, using a different construction that exploits 
a randomization procedure on certain basic elements of the flow.
\end{parag}

%
%  ACKNOWLEDGMENTS
%
\begin{parag*}[Acknowledgments]
The first and third authors acknowledge the hospitality of the
Department of Mathematics and Computer Science at the University
of Basel, where this work was started.
Their stay was partially supported by the Swiss National Science
Foundation grants 140232 and 156112.
The visits of the second author to Pisa were supported by the
University of Pisa PRA project ``Metodi variazionali per problemi
geometrici [Variational Methods for Geometric Problems]''.
The second author was partially supported by the ERC Starting 
Grant 676675 FLIRT and third author by the US~National Science
Foundation grants DMS~1312727 and 1615457.
\end{parag*}

%
%	SECTION 2
%
\section{Preliminaries}\label{s:prelim}

Throughout the paper, we will make extensive use of homogeneous
Sobolev spaces with real order of differentiability and of their properties. We
present here their definition and main properties of interest for our
work, namely those regarding scaling, interpolation, and embeddings.
For a systematic exposition we refer the reader to~\cite{bergh,chemin,valdinoci,grafakos,triebel}.
In addition, in the last part of this section, we define the two notions of mixing scale
that we will use in our work.

\medskip

We limit our presentation to the two-dimensional case, however all definitions and
results can be extended with obvious changes to the case of higher space
dimensions.
We work both on the plane $\R^2$ and on the two-dimensional flat torus
$\T^2 := \R^2 / \ZZ^2$.
The Fourier transform of a
tempered distribution $f$ on $\R^2$ is denoted by $\hat f(\xi)$;
the Fourier coefficients of a distribution $f$ on $\T^2$
are denoted by $\hat{f}(k)$.

\begin{parag}[Homogeneous Sobolev spaces $\boldsymbol{\dot{H}^s}$]
\label{ss:defH}
For $s \in \R$, we say that a distribution $f$ on $\T^2$
belongs to the homogeneous Sobolev space $\dot{H}^s(\T^2)$ if
\begin{equation}\label{e:defHT}
\| f \|^2_{\dot{H}^s(\T^2)} 
:= \sum_{k \in \ZZ^2} |k|^{2s} |\hat{f}(k)|^2 
< \infty 
\,.
\end{equation}

We say that a tempered distribution $f$ on $\R^2$
belongs to the homogeneous Sobolev space
$\dot{H}^s(\R^2)$ if $\hat f \in L^1_{\rm loc}(\R^2)$ and
\begin{equation}\label{e:defH}
\| f \|^2_{\dot{H}^s(\R^2)} 
:= \int_{\R^2} | \xi |^{2s} | \hat{f}(\xi) |^2 \, d\xi 
< \infty 
\,.
\end{equation}
This definition is suitable for our purposes, 
though it is different from the standard definition, 
which requires the use of equivalence classes modulo polynomials 
(see e.g.~\cite{grafakos,triebel}).
\end{parag}

We remark that homogeneous Sobolev spaces do not form a scale, due to the singularity
of the multiplier at the origin in frequency space. In particular, it is
generally not true that any square integrable function is automatically in
$\dot{H}^s$, for~$s<0$.

\begin{remark}\label{r:fctsinH}
\quad
%%%%%
\begin{enumerate}
[label=(\roman*), ref=(\roman*),
leftmargin=0pt, itemsep=2pt, itemindent=30pt]

\item
\label{r:fctsinH:1} 
From \eqref{e:defHT} we immediately recognize that, in order for 
a function $f \in L^2(\T^2)$ to belong to some $\dot{H}^{s}(\T^2)$ 
with $s<0$, it is necessary that~$\hat{f}(0)=0$. 
This corresponds to the zero-integral condition $\int_{\T^2} f = 0$. 
Conversely, let $f \in L^2(\T^2)$ be a function with zero integral: 
since the sequence of its Fourier coefficients 
$\{ \hat{f}(k) \}_{k \in \ZZ^2}$ belongs to $\ell^2$, 
we deduce that such a function necessarily belongs to 
$\dot{H}^{s}(\T^2)$ for every $s \leq 0$. 

\item 
\label{r:fctsinH:2} 
Given a function $f \in L^1(\R^2)$, its Fourier transform
$\hat{f}$ is continuous.
If~${s \leq -1}$ the singularity at $\xi=0$ in \eqref{e:defH}
is not integrable, unless $\hat{f}(0)=0$. 
This means that the  condition $\int_{\R^2} f = 0$ is 
a necessary condition for $f$ to belong to 
$\dot{H}^{s}(\R^2)$ with $s \leq -1$. 

\item 
\label{r:fctsinH:3} 
Let $f \in L^2(\R^2)$ have compact support and zero integral. 
Paley-Wiener theorem implies that the Fourier transform $\hat{f}$ 
is analytic. 
In particular, there is a constant $C>0$ for which 
$|\hat{f}(\xi)|=| \hat{f}(\xi) - \hat{f}(0)| \leq C |\xi|$
for any~${|\xi|\leq 1}$, therefore the singularity
at $\xi=0$ in \eqref{e:defH} is integrable for every~${s > -2}$. 
Since~${\hat{f} \in L^2(\R^2)}$, we conclude 
that~$f \in \dot{H}^s(\R^2)$ for any $-2<s \leq 0$.
\end{enumerate}
\end{remark}

\begin{parag}[Homogeneous Sobolev spaces $\boldsymbol{\dot{W}^{s,p}}$]
In the particular case $s \geq 0$ we extend the definition in 
\S\ref{ss:defH} to an arbitrary summability exponent $1<p<\infty$.%
\footnoteb{Only the spaces $\dot{W}^{s,p}$ with $s\geq 0$ will 
be needed for our scopes (see again~\cite{grafakos,triebel} 
for a discussion of these spaces with regularity index $s\in\R$).}
We say that a distribution $f$ on $\T^2$
belongs to the homogeneous Sobolev space $\dot{W}^{s,p} (\T^2)$ if
\begin{equation}\label{e:defWT}
\sum_{k \in \ZZ^2} |k|^s \hat{f}(k) e^{ikx} \in L^p(\T^2) \,,
\end{equation}
and we let $\smash{ \| f \|_{\dot{W}^{s,p}(\T^2)} }$ be the $L^p(\T^2)$ 
norm of the function in~\eqref{e:defWT}. 
We observe that this definition gives a seminorm and not a norm, in general.

We say that a tempered distribution $f$ on $\R^2$
belongs to the homogeneous Sobolev space $\dot{W}^{s,p}(\R^2)$ 
if $\hat{f} \in L^1_{\rm loc}(\R^2)$ and
\begin{equation}\label{e:defW}
\mathscr{F}^{-1} \big( | \xi |^s \hat{f}(\xi) \big) \in L^p(\R^2)
\end{equation}
where $\mathscr{F}^{-1}$ is the inverse of the Fourier transform,
and we let $\smash{ \| f \|_{\dot{W}^{s,p}(\R^2)} }$ be the $L^p(\R^2)$ 
norm of the function in~\eqref{e:defW}. 

The condition that $\smash{\hat{f} \in L^1_{\rm loc}(\R^2)}$
guarantees that this quantity is a norm.  
We have the obvious identification $\dot W^{s,2}\equiv \dot H^s$.
\end{parag}

In our work, homogeneous spaces will be used only to measure the 
``size'' of given functions and velocity fields, which will be 
typically regular. 
We can avoid to give a rigorous and complete definition of these 
spaces, which again is based on equivalence of distributions modulo
polynomials, as we can just rely on the seminorms defined above. 
(We refer to e.g.~\cite{grafakos,triebel} for a more detailed 
discussion of these spaces.)

\begin{remark}\label{r:tay}
\quad
%%%%%
\begin{enumerate}
[label=(\roman*), ref=(\roman*), 
leftmargin=0pt, itemsep=2pt, itemindent=30 pt]

\item
%\label{r:tay:1} 
The non-homogeneous Sobolev spaces $H^s$ and $W^{s,p}$ are defined 
by replacing in~\eqref{e:defHT}, \eqref{e:defH}, \eqref{e:defWT}, 
and~\eqref{e:defW} the $|k|$ and $|\xi|$ by the symbols 
$\langle k \rangle := \sqrt{1+|k|^2}$ and 
$\langle \xi \rangle := \sqrt{1+|\xi|^2}$, respectively.
If $s\in\NN$ then~$H^s$ and $W^{s,p}$ coincide with the usual 
Sobolev spaces defined using weak derivatives. 

\item
%\label{r:tay:2} 
In the case of the plane $\R^2$, if we consider functions that 
are supported in a fixed compact set then the homogeneous and 
the non-homogeneous norms are equivalent, and therefore the 
homogeneous and the non-homogeneous spaces coincide. 
In the case of the torus $\T^2$, the homogeneous and the 
non-homogeneous spaces always coincide.
The non-homogeneous norm is equivalent to the sum of the 
homogeneous norm and  the $L^2$ norm. 

\item
%\label{r:tay:3} 
In contrast to the case of negative spaces (see
Remarks~\ref{r:fctsinH}\ref{r:fctsinH:1}, \ref{r:fctsinH}\ref{r:fctsinH:2}), 
it is not necessary for functions to have zero
average to belong to $\dot{W}^{s,p}$ with $s \geq 0$. 

\item
\label{r:tay:4} 
Let $K$ be a  compact set contained in the open square~$\Qone = (-1/2,1/2)^2$. 
Let $\sigma$ be the canonical projection of the plane $\R^2$ onto 
the torus~$\T^2 = \R^2 / \ZZ^2$. 
The restriction of $\sigma$ to $(-1/2,1/2)^2$ is a diffeomorphism. 
A Sobolev function $f$ on $\R^2$ with support contained in $K$ can be 
identified with a Sobolev function $\tilde f$ on $\T^2$ via the formula 
$\tilde f = f \circ \sigma^{-1}$.
It can be shown that, if~$s \geq 0$ and $1< p <\infty$, then
\[
C^{-1} \| f \|_{\dot{W}^{s,p}(\R^2)}
\leq \| \tilde f \|_{\dot{W}^{s,p}(\T^2)} 
\leq C \| f \|_{\dot{W}^{s,p}(\R^2)} 
\, ,
\]
where the constant $C$ depends on $s$, $p$, and on the compact set $K$
(see~\cite{taylor}).
\end{enumerate}
\end{remark}

\begin{parag}[Lipschitz-H\"older spaces]
\label{ss:lhs}
For notational convenience, in this paper we denote 
by $\dot{W}^{s,\infty}(\T^2)$ and $\dot{W}^{s,\infty}(\R^2)$ 
the homogeneous Lipschitz-H\"older spaces
defined as follows (we do not write explicitly the domain).

If $s$ is a positive integer,
we say that a function $f$ belongs to 
$\dot{W}^{s,\infty}$
if $f\in C^{s-1}$ and there is a constant $C>0$ such that
\begin{equation}
\label{e:LIP}
| f^{(s-1)}(x) - f^{(s-1)}(y)| \leq C |x-y|
\quad\text{for every $x$ and $y$.}
\end{equation}
We let $\| f \|_{\dot{W}^{k,\infty}}$ be
the minimal constant $C$ for which \eqref{e:LIP} holds.

If $s\geq 0$ is not an integer, we let $\lfloor s \rfloor$ 
be the largest integer smaller than $s$, 
and we say that a function $f$ belongs to $\dot{W}^{s,\infty}$
if $f \in C^{\lfloor s \rfloor}$ and there is a constant $C>0$
such that
\begin{equation}\label{e:HOLD}
| f^{(\lfloor s \rfloor)}(x) - f^{(\lfloor s \rfloor)}(y) | 
\leq C |x-y|^{s - \lfloor s \rfloor}
\quad\text{for every $x$ and $y$.}
\end{equation}
We let $\| f \|_{\dot{W}^{s,\infty}}$  be
the minimal constant $C$ for which \eqref{e:HOLD} holds.
\end{parag}

\begin{remark}
\label{r:embedding}
In two space dimensions, the Sobolev space $\dot{W}^{s,p}$ embeds in the
Lipschitz space $\dot{W}^{1,\infty}$ if $s>1$ and $p > \frac{2}{s-1}$.
\end{remark}

\begin{parag}[Scaling properties]
We will be frequently interested in the behavior of homogeneous Sobolev norms
under rescaling. Given $\lambda>0$ and a function $f$, we set
\begin{equation}\label{e:frisc}
f_\lambda(x) := f \left( \frac{x}{\lambda} \right) \,.
\end{equation}

If $f$ is defined on the torus and $1/\lambda$ is an integer, then the function
$f_\lambda$ in \eqref{e:frisc} is well defined on the torus and it holds
\begin{equation}\label{e:scaleT}
\| f_\lambda \|_{\dot{W}^{s,p}(\T^2)}
=
\lambda^{-s} \, \| f \|_{\dot{W}^{s,p}(\T^2)} \,.
\end{equation}

If $f$ is defined on the plane, then the function $f_\lambda$ in \eqref{e:frisc}
is well defined for any~$\lambda>0$ and it holds
\begin{equation}\label{e:scaleR}
\| f_\lambda \|_{\dot{W}^{s,p}(\R^2)}
=
\lambda^{\frac{2}{p}-s} \, \| f \|_{\dot{W}^{s,p}(\R^2)} \,.
\end{equation}

When $p=2$, both formulas~\eqref{e:scaleT} and~\eqref{e:scaleR}
hold also for $s<0$.
\end{parag}

\begin{remark}
The difference between the exponents in formulas~\eqref{e:scaleT}
and~\eqref{e:scaleR} is due to the fact that, in the case of the 
torus, we are not changing the period, hence the measure of the 
torus, when rescaling.
In particular, the rescaling of a single bump on the plane remains 
a single bump, while on the torus $1 / \lambda^2$ rescaled copies 
of the bump are produced.
\end{remark}

\begin{parag}[Interpolation]
We will frequently rely on the following standard interpolation inequality.
If $s_1 < s < s_2$ and 
$s = \vartheta s_1 + (1-\vartheta) s_2$, $\vartheta \in (0,1)$, then
\begin{equation}\label{e:interpH}
\| f \|_{\dot{H}^s} \leq \| f \|_{\dot{H}^{s_1}}^\vartheta \| f
\|_{\dot{H}^{s_2}}^{1-\vartheta}
\end{equation}
and
\begin{equation}\label{e:interpW}
\| f \|_{\dot{W}^{s,p}} \leq \| f \|_{\dot{W}^{s_1,p}}^\vartheta \| f
\|_{\dot{W}^{s_2,p}}^{1-\vartheta} \,,
\end{equation}
where both inequalities hold both with domain $\T^2$ and $\R^2$.
The same holds in the case $p=\infty$ in the context of Lipschitz-H\"older
spaces
(recall~\S\ref{ss:lhs}) and can be proven with a simple direct argument.
\end{parag}

We next introduce the two notions of mixing scale that will be employed 
in this paper to quantify the level of mixedness of the solution $\rho$. 
Both definitions can be given on the torus and on the plane, and will 
be stated for functions of space variables only with zero integral.
In our framework, the mixing scales of the solution will of course depend 
on time, due to the fact that the solution is dependent on time.

\begin{parag}[Functional mixing scale~\cite{doering,MMP}]
\label{def:Functscale}
Assume that $\rho$ has zero integral. 
The functional mixing scale of $\rho$ is $\| \rho\|_{\dot{H}^{-1}}$.
\end{parag}

\begin{parag}[Geometric mixing scale~\cite{bressan}]
\label{d:geomix}
Assume that $\rho$ has zero integral. 
Given $0 < \kappa < 1$, the geometric mixing scale of $\rho$ is the
infimum of all $\varepsilon>0$ such that, for every $x \in
{\mathbb T}^2$, there holds
\begin{equation}\label{e:geom}
\frac{1}{\| \rho \|_\infty}
\left| \mint_{B_\varepsilon(x)} \rho \, dy \right| \leq \kappa \,.
\end{equation}
\end{parag}
The parameter $\kappa$ is fixed and plays a minor role in the definition.
Informally, in order for $\rho$ to have geometric mixing scale 
$\varepsilon$, the average of the solution on every ball of radius 
$\varepsilon$ is essentially zero.
Alternatively, the property of having (approximately) zero
average needs to be localizable to balls of radius $\varepsilon$.

\begin{remark} 
The geometric mixing scale has been originally introduced
in~\cite{bressan} for solutions with value $\pm 1$: 
given $0<\tilde \kappa<1/2$, \eqref{e:geom} is replaced by the
requirement that
\begin{equation}\label{e:geomgeom}
\tilde \kappa \leq \frac{|\{ \rho = 1 \} \cap B_\varepsilon(x) |}{|
B_\varepsilon(x)|}
\leq 1 - \tilde \kappa \,.
\end{equation}
Informally,  in order for $\rho$ to have geometric mixing scale~$\varepsilon$,
every ball of radius $\varepsilon$ contains a ``substantial portion'' 
of both level sets $\{ \rho = 1 \}$ and $\{ \rho = -1 \}$.
The more general definition we adopt (see Definition~\ref{d:geomix}) 
has been introduced in~\cite{zlatos} and it applies to every bounded 
solution $\rho$, without any constraint on its values.
It is easily seen that~\eqref{e:geom} and~\eqref{e:geomgeom}
correspond if $\kappa = 1 - 2 \tilde \kappa$.
\end{remark}

As previously mentioned, the two notions of mixing scale are not equivalent,
though they are strongly related (we refer to \cite{llnmd,zillinger} 
for a further discussion on this point; see also Lemma~\ref{l:fg}).

%
%	SECTION 3
%
\section{Scaling analysis in a self-similar construction}
\label{selfsimilar}

A conceivable procedure for mixing consists of a self-similar evolution.
Such a procedure, together with the related scaling analysis, 
has been presented in~\cite{ACM}.
We work on the torus $\T^2$. We let $s \geq 0$ and $1 \leq p \leq \infty$
be fixed and we make the following assumption.

\begin{parag}[Assumption: self-similar base element]
\label{ass}
There exist a velocity field~$u_0$ and a (not identically zero) solution
$\rho_0$ to \eqref{e:continuity}, both defined for $0 \leq t \leq 1$ and~$x \in
\T^2$, such that:
%%%%%
\begin{enumerate}
[label=(\roman*), ref=(\roman*),
itemsep=2pt, leftmargin=30 pt]

\item 
$u_0$ is  bounded, bounded in $\dot{W}^{s,p} ({\mathbb T}^2)$ uniformly in
time, and divergence-free;

\item 
$\rho_0$ is bounded and has mean zero for all times;

\item\label{ass:iii} 
there exists a positive constant $\lambda$, with $1/\lambda$ an
integer greater or equal than $2$, such that
\[
\rho_0(1, x) = \rho_0 \left( 0, \frac{x}{\lambda} \right) \,.
\]
\end{enumerate}
\end{parag}

An explicit example of a $u_0$ and a $\rho_0$ satisfying these 
assumption will be given in Section~\ref{pinching} 
for $s=1$ and arbitrary $1\leq p<\infty$.
In fact, the range of indices for this example is slightly larger 
(see~\eqref{e:range}). 
However, it is not evident to us  how to construct an example that 
satisfies Assumption~\ref{ass} outside the range in~\eqref{e:range},
in particular for the case $s=1$ and $p=\infty$. 
This limitation leads us to introduce the second geometric construction
in~\S\ref{quasiselfsimilar}.

\medskip

For later use, we introduce the following definition.

\begin{definition}\label{d:tiling}
Given $\lambda > 0$, with $1/\lambda$ an integer, 
we denote by $\Til_\lambda$ the tiling of $\mathbb{T}^2$ 
consisting of $1/\lambda^2$ open squares of side-length 
$\lambda$ in $\mathbb{T}^2$  of the form
\[
\left\{ (x,y) \in {\mathbb T}^2 \, : \, 
\text{$(k-1)\lambda < x < k\lambda$ and $(h-1)\lambda < y < h\lambda$} 
\right\}
\,,
\]
with $k,h = 1,2,\ldots, 1 / \lambda$.

Denoting by $\Qone$ the unit open square
$(-1/2,1/2)^2 \subset \mathbb{R}^2$, the tiling $\Til_\lambda$ 
of $\Qone$ is defined in a similar way. 
Given any square $Q \in \Til_{\lambda}$, we denote by $r_Q$ its center, 
so that $Q = \lambda \Qone + r_Q$.
\end{definition}

\begin{parag}[A self-similar construction]\label{ss:ssc}
We begin by fixing a positive number $\tau$ (to be determined later).
Under Assumption \ref{ass}, for each integer $n=1,2,\ldots$ and 
for $t \in [0,\tau^n]$ we set
\[
u_n(t,x) := \frac{\lambda^n}{\tau^n} \; u_0 \left( \frac{t}{\tau^n} ,
\frac{x}{\lambda^n} \right) \,,
\quad \rho_n(t,x) := \rho_0 \left( \frac{t}{\tau^n} , \frac{x}{\lambda^n}
\right) \,.
\] 
Then $\rho_n$ is a solution of~\eqref{e:continuity} corresponding
to the velocity field $u_n$. 
Moreover, because of Assumption~\ref{ass}\ref{ass:iii},
\begin{equation}\label{e:glue}
\rho_n (\tau^n , x) = \rho_{n+1} (0,x) \,.
\end{equation}

We now define $u$ and $\rho$ by concatenating the velocity fields
$u_0, u_1, \ldots$ and the corresponding solutions
$\rho_0, \rho_1, \ldots$. In detail, we let
\[
u(t,x) := u_n (t - T_n , x) 
\,, \quad 
\rho(t,x) := \rho_n (t - T_n , x)
\]
for $T_n \leq t < T_{n+1}$, and $n=1,2,\ldots$, where
\[
T_n := \sum_{i=0}^{n-1} \tau^i 
\quad \text{for $n=1, 2, \ldots, \infty$.}
\]
With this choice,  $u$ and $\rho$ are defined for $0 \leq t < T_\infty$. 
Moreover, it follows from Equation~\eqref{e:glue} that $\rho$ is a weak 
solution on $(0,T_\infty)$ of the Cauchy
problem for~\eqref{e:continuity} with velocity field $u$ and
initial condition $\bar{\rho}(x) = \rho_0(0,x)$.

\begin{figure}[ht]
\begin{center}
  \includegraphics[scale=1]{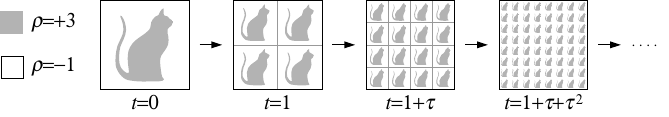}
  \caption{Example of self-similar evolution for a function
$\rho$ taking only two values.}
  \label{f:self}
\end{center}
\end{figure}

Using~\eqref{e:scaleT}, we compute
\begin{equation}\label{e:fsc}
\left\| u_n (t,\cdot) \right\|_{\dot{W}^{s,p}(\mathbb{T}^2)}
= \left( \frac{\lambda^{1-s}}{\tau} \right)^n \left\| u_0 \left(
\frac{t}{\tau^n} , \cdot \right) \right\|_{\dot{W}^{s,p}(\mathbb{T}^2)} \,.
\end{equation}
Next we choose
\[ 
\tau = \lambda^{1-s}
\, , 
\]
so that $u$
is bounded in $\dot{W}^{s,p}({\mathbb T}^2)$ uniformly in time.
Moreover,
\begin{equation}
\label{e:ssc}
\| \rho_n (t,\cdot) \|_{\dot{H}^{-1}(\mathbb{T}^2)} =  \lambda^n \left\| \rho_0
\left( \frac{t}{\tau^n} , \cdot \right) \right\|_{\dot{H}^{-1}(\mathbb{T}^2)}
\leq M \lambda^n 
\,,
\end{equation}
where we have set
\[
M := \sup_{0 \leq t \leq 1} \| \rho_0 (t,\cdot) \|_{\dot{H}^{-1}(\mathbb{T}^2)}
\, .
\]
Equivalently, 
\begin{equation}\label{e:decay}
\| \rho(t,\cdot) \|_{\dot{H}^{-1}(\mathbb{T}^2)} \leq M \lambda^n 
\quad\text{for $T_n \leq t < T_{n+1}$.}
\end{equation}
We have three possible cases (recall that $s\geq0$):
%%%%%
\begin{enumerate}
[label=(\alph*),
leftmargin=30pt, itemsep=2pt]

\item 
$s<1$, hence $\tau < 1$: In this case, $T_\infty$ is
finite and
\[
\| \rho(t,\cdot) \|_{\dot{H}^{-1}(\T^2)} \to 0 
\,,
\]
as $t \to T_\infty$. That is, we have perfect mixing in finite time.

\item 
$s=1$, hence $\tau = 1$. In this case, $T_\infty = \infty$, $T_n =
n$, and the inequality $t < T_{n+1}$ in~\eqref{e:decay} becomes $t-1<n$.
The estimate in~\eqref{e:decay} then  yields the following exponential decay of
the functional mixing scale:
\[
\| \rho(t,\cdot) \|_{\dot{H}^{-1}(\mathbb{T}^2)} \leq M \lambda^{t-1} \,.
\]

\item 
$s>1$, hence $\tau>1$. In this case $T_\infty = \infty$ and
\[
T_n = \frac{\tau^n-1}{\tau - 1} = \frac{\lambda^{(1-s)n}-1}{\lambda^{1-s} - 1}.
\]
By the same argument as above, \eqref{e:decay} implies the following polynomial
decay of the functional mixing scale:
\[
\| \rho(t,\cdot) \|_{\dot{H}^{-1}(\mathbb{T}^2)} \leq M \frac{\left[ 1 + t
(\lambda^{1-s} - 1) \right]^{-\frac{1}{s-1}}}{\lambda}
\simeq C(M,\lambda,s) \, t^{-\frac{1}{s-1}} \,.
\]
\end{enumerate}
\end{parag}

We formalize the above discussion in the following theorem.

\begin{theorem}\label{t:main}
Given $s\geq 0$ and $1 \leq p \leq \infty$, under Assumption~\ref{ass}, there
exist a bounded divergence-free velocity field $u$ and a weak solution $\rho$ of
the Cauchy problem for~\eqref{e:continuity}, such that $u$ is bounded in
$\dot{W}^{s,p}(\mathbb{T}^2)$ uniformly in time and the functional mixing scale
of $\rho$ exhibits the following behavior depending on $s$:
%%%%%
\begin{itemize}
[leftmargin=30pt, itemsep=0pt]

\item 
case $s<1$: perfect mixing in finite time;
\item 
case $s=1$: exponential decay;
\item 
case $s>1$: polynomial decay.
\end{itemize}
\end{theorem}

In fact, all homogeneous negative Sobolev norms $\|\rho(t)\|_{\dot{H}^{-r}}$
(with $r>0$) would exhibit the same behavior, the only difference being in the constant
for the exponential decay and the exponent for the polynomial decay, which
depend on~$r$.
We observe that, in the case $s>1$, such self-similar scaling analysis does
not match the exponential lower bound for the (geometric and functional)
mixing scale, which is expected to be optimal 
(recall the discussion in~\S\ref{ss:literature}\ref{ss:literature:d}).

\medskip

The following lemma shows that the geometric mixing scale exhibits 
the same behavior as the functional mixing scale, as established 
in Theorem~\ref{t:main} above.

\begin{lemma}\label{l:fg}
Fix $n\in\NN$, and let $\rho$ be a bounded function
such that
\begin{equation}\label{e:zerotil}
\mint_{Q} \rho \, dy = 0,
\end{equation}
for every square $Q \in \Til_{\lambda^n}$. Then the geometric mixing scale of $\rho$, 
introduced  in Definition~\ref{d:geomix}, is at most
\begin{equation}\label{e:peri}
\frac{4 \sqrt{2} \, \lambda^n}{\kappa}\,.
\end{equation}
\end{lemma}

\begin{proof}
We fix an arbitrary ball $B_\varepsilon(x)$. Using \eqref{e:zerotil} we can
estimate
\begin{align*}
\frac{1}{\| \rho\|_\infty} \left| \int_{B_\varepsilon(x)} \rho \, dy \right|
& \leq \frac{1}{\| \rho\|_\infty}
\sum_{\substack{Q \in \Til_{\lambda^n}  \\ \text{ $Q \cap
\bd B_\varepsilon(x) \ne\varnothing$}}}
\int_Q | \rho| \, dy \\
& \leq
\frac{1}{\| \rho\|_\infty}
\int_{B_{\varepsilon+\sqrt{2}\lambda^n}(x)\setminus B_{\varepsilon-\sqrt{2}\lambda^n}(x)} 
|\rho|\, dy \\
&\leq
\pi \big[ (\varepsilon+\sqrt{2}\lambda^n)^2 - (\varepsilon-\sqrt{2}\lambda^n)^2 \big]
=
4 \sqrt{2} \pi  \varepsilon \lambda^n\,,
\end{align*}
where the second inequality follows from elementary geometric considerations.
Hence,
\[
\frac{1}{\| \rho\|_\infty}  \left| \mint_{B_\varepsilon} \rho \, dx \right|
\leq
\frac{4 \sqrt{2} \,\lambda^n }{\varepsilon} 
\,,
\]
and the right-hand side is less or equal than $\kappa$
for every $\varepsilon$ greater or equal than
the quantity in \eqref{e:peri}. Therefore, from \eqref{e:geom} the
desired estimate on the geometric mixing scale follows.
\end{proof}

\begin{parag}[Regularity in time]\label{ss:timereg}
Under Assumption~\ref{ass}, the self-similar construction described above
ensures Sobolev regularity of the velocity field with respect to the
space variable, uniformly in time. No regularity with respect to the 
time variable is provided.

However, in all examples presented in this paper, the velocity field
is smooth in space and piecewise smooth in time.
If the velocity field is smooth in time on two adjacent time intervals,
and if it can be smoothly extended to the closure of each of them,
then the discontinuity across the interface of the two intervals can be eliminated
by a suitable reparametrization of time. More precisely, we replace
in each time interval $u$ and $\rho$ by
\[
\tilde{u}(t,x) := \eta'(t) \, u \big( \eta(t),x \big)
\, , \quad
\tilde{\rho}(t,x) := \rho \big( \eta(t),x \big)
\, ,
\]
where in each interval the smooth function $\eta$ is chosen to be increasing,
surjective, and constant in a
small (left or right) neighborhood of each endpoint of the interval.
It is immediate to check that $\tilde \rho$ solves the Cauchy problem 
for~\eqref{e:continuity} with velocity field $\tilde u$, that $\tilde u$ 
is smooth on the union of the closures of the two time intervals, and that 
the value of the solution at the endpoints of both
intervals has not changed.

We remark  that the argument above does not apply in case the velocity field
lacks a smooth extension to the closure of the time intervals. 
In this case the time discontinuity cannot be eliminated. 
This is indeed the case for the example presented in
Section~\ref{pinching}. The time singularity cannot be avoided there, 
given that the topological properties of smooth sets are not preserved
along the time evolution realized in that example.
\end{parag}

%
%	SECTION 4
%
\section{First geometric construction}
\label{basic}

In this section, we establish a geometric lemma that  is at the core 
of the construction of optimal mixers in our work.
More precisely, in Proposition~\ref{s:basic} below, we show that,
given a regular set $E$ in the plane that evolves smoothly in
time, we can construct a smooth, divergence-free velocity field $u$
such that the characteristic function of $E$ solves the continuity
equation \eqref{e:continuity} associated to $u$.

\medskip
We begin by introducing some notation.
Given a vector $v=(v_1,v_2)\in\R^2$, we denote by
$v^\perp$ the vector obtained by rotating  $v$  counter clockwise by $90^\circ$,
that is,
\[
v^\perp:=(-v_2,v_1)
\, .
\]
Given a set $E$ in $\R^2$ and a point $x\in\R^2$,
we denote  the distance of $x$
from $E$ by $\mathrm{dist}(x,E)$, namely:
\[
     \mathrm{dist}(x,E) := \inf\big\{|x-y| \, \mid \, y \in E\big\}.
\]
If there exists exactly one point $y\in E$
where such infimum is attained, this point will be
called the \emph{projection} of $x$ onto $E$ and
denoted by $p_E(x)$.
For every $r>0$, we shall also denote
 the open $r$-neighborhood of $E$ by $B(E,r)$:
\[
B(E,r) := \big\{ x\in\R^2 \, : \ \mathrm{dist}(x,E)<r \big\}
\, .
\]
We discuss next various notions of paths, which will be needed 
for the geometric construction. We consider only two kinds of paths.

\begin{parag}[Paths and curves]
\label{s:paths}

A \emph{closed path} is a continuous map
$\gamma=\gamma(s)$ from the circle,
which we identify with the one-dimensional
torus $\T^1 := \R/\mathbb{Z}$,
to the plane~$\R^2$. We require
that $\gamma$ is injective, of class $C^1$,
and satisfies $\dot\gamma(s)\ne 0$ for all~$s\in\T^1$.
A \emph{closed (oriented) curve} is the image $\Gamma=\gamma(\T^1)$
of a closed path $\gamma$.

A \emph{proper path} is a continuous map
$\gamma=\gamma(s)$ from the the real line $\R$
to the plane $\R^2$ which is proper, that is,
$|\gamma(s)|$ tends to $+\infty$ as $s\to\pm\infty$.
As before, we require that $\gamma$ is injective,
of class $C^1$, and satisfies $\dot\gamma(s)\ne 0$
for all $s\in\R$.
A~\emph{proper (oriented) curve} is the image $\Gamma=\gamma(\R)$
of a proper path $\gamma$.

When it is not necessary to distinguish between closed
and proper paths (or curves), we will simply refer to them 
as a path (or a curve), and denote the parametrization domain, 
which is either $\T^1$ or $\R$, by the letter $J$.
As usual, the regularity of a curve $\Gamma$
refers to the regularity of the parametrization $\gamma$.

Let $\Gamma$ be a curve parametrized by $\gamma$.
A \emph{sub-arc} of $\Gamma$ is any set of the form
$\gamma(J')$ where $J'$ is an interval contained in $J$;
a sub-arc is \emph{proper} if it is strictly
contained in $\Gamma$.

The unit tangent vector $\tau(x)$ and
the unit normal vector $\eta(x)$
at a point $x=\gamma(s)$ in $\Gamma$ are given by\,%
\footnoteb{Thanks to the minus sign in the definition of the unit normal
vector, if $\gamma$ is a counter-clockwise
parametrization of the boundary of an open set then $\eta$ coincides with
the outer normal to the boundary of the set. }
\[
\tau:= \dot\gamma / |\dot\gamma|
\, , \quad
\eta:= -\tau^\perp = - \dot\gamma^\perp / |\dot\gamma|
\, .
\]
In particular if $|\dot\gamma(s)|$ is equal to
the constant $\ell$ for all $s$ then
$\tau= \dot\gamma / \ell$,
$\eta= -\dot\gamma^\perp / \ell$, and
 the curvature $\kappa(x)$ of $\Gamma$ at
the point $x=\gamma(s)$ satisfies the equation
\[
-\kappa\, \eta = \ddot\gamma/\ell^2
\, .
\]

The \emph{tubular radius} of $\Gamma$
is the largest $r\ge 0$
such that the map $\Psi$ given by\,%
\footnoteb{With a slight abuse of notation, we sometimes
write the geometric quantities $\tau$, $\eta$ and $\kappa$
as functions of the parametrization variable $s$ instead of $x$.}
\begin{equation}
\label{e:param}
\Psi: (s,y) \mapsto \gamma(s) + y \, \eta(s)
\end{equation}
is injective on $J\times(-r,r)$.

If $\Gamma$ is of class $C^2$ then the tubular radius
is smaller than the \emph{curvature radius}
$1/|\kappa(x)|$ for every $x\in\Gamma$.

If $\Gamma$ is of class $C^k$ with
$k\ge 2$ and the tubular radius $r$ is strictly
positive, the map~$\Psi$ is a diffeomorphism
of class $C^{k-1}$ from $J\times(-r,r)$
to the tubular neighborhood~$B(\Gamma,r)$,
the projection $p_\Gamma(x)$ is well-defined
for every point~${x=\Psi(s,y)}$ in $B(\Gamma,r)$
and agrees with $\gamma(s)$.

If $\Gamma$ is closed and of class $C^2$, then the
tubular radius is strictly positive.
\end{parag}

\begin{parag}[Time-dependent paths and curves]
\label{s:tpaths}
Throughout the paper,  we often consider paths and curves that 
depend on time.
In this case, $\gamma$ is a map from the product $I\times J$
to the plane $\R^2$, where $I$ is a time interval (which could be 
open, closed, or neither), and $\Gamma$ is a map that assigns 
a curve $\Gamma(t)$ in $\R^2$ to every $t\in I$. 
The regularity of these paths and curves is then intended 
as the regularity of $\gamma$ in both variables.

In what follows, we reserve the letter~$t$ for the time
variable in $I$ and the letter~$s$ for the parametrization 
variable in $J$. 
Correspondingly, we write  $\bd_t\gamma$ for the partial derivative
with respect to~$t$ and $\dot\gamma$ for the partial derivative
with respect to~$s$.

The \emph{normal velocity} $v_\nn=v_\nn(t,x)$ of $\Gamma$
at time $t$ and at the point $x=\gamma(t,s)$ is the
normal component of the vector $\bd_t\gamma(t,s)$,
that is,
\[
v_\nn := \bd_t\gamma \cdot \eta
\, .
\]
We note that the normal velocity does not change under
strictly increasing reparametrizations of $\gamma$ 
in the variable $s$.

\end{parag}

\begin{parag}[Time-dependent domains]
\label{s:geomevol}
A time-dependent domain is a map $E$ that assigns an open 
subset $E(t)$ of $\R^2$ to every time $t$ in the interval $I$.
We say that~$E$ is of class $C^k$, if there exist
finitely many time-dependent curves $\Gamma_i$,
parametrized by paths
$\gamma_i:I\times J_i\to\R^2$ of class $C^k$,
such that for every  $t\in I$ the boundary $\bd E(t)$
can be written as disjoint union of
the curves $\Gamma_i(t)$.

We also require
that each parametrization is \emph{counter-clockwise},
which means that the normal vector $\eta$
defined in \S\ref{s:paths} agrees with the \emph{outer normal}
to the boundary $\bd E(t)$ at every time $t$ and at every
point $x=\gamma_i(t,s)$.
Thus the normal velocity $v_\nn$ defined in \S\ref{s:tpaths}
agrees with the \emph{outer normal velocity} of $\bd E(t)$.
\end{parag}

\begin{parag}[Compatible velocity fields]
\label{s:compatible}
Let $u$ be a time-dependent velocity field on~$\R^2$ of class $C^1$. 
We say that $u$ is \emph{compatible} with a time-dependent curve
$\Gamma$ if, for every time $t$ and  every
point $x\in\Gamma(t)$, the normal
velocity $v_\nn$ of $\Gamma$ agrees with
the normal component of $u$,
that is
\begin{equation}
\label{e:compat}
v_\nn = u \cdot \eta
\, .
\end{equation}
Accordingly, we say that $u$
is \emph{compatible} with a time-dependent domain $E$ of
class~$C^1$ if the normal component of $u$ agrees with
the outer normal velocity $v_\nn$ of $E$ at every time $t$
and at every point $x\in\bd E(t)$.

Given $t_0\in I$, we let $\{\Phi(t,\cdot): \, t\in I\}$ be the flow
associated to $u$ with initial time~$t_0$,  which means that
each $\Phi(t,\cdot)$ is an homeomorphism from $\R^2$ into $\R^2$,
and that for every $x_0\in\R^2$ the map $t\mapsto \Phi(t,x_0)$
solves the ordinary differential equation $\dot x=u(t,x)$
with initial condition $x(t_0)=x_0$.
Then the compatibility of $u$ and $\Gamma$ implies that
$\Gamma(t)=\Phi(t,\Gamma(t_0))$ for every $t\in I$. 
Similarly,  the compatibility of $u$ and $E$ implies that
$\bd E(t)=\Phi(t,\bd E(t_0))$, and consequently that
\[
E(t)=\Phi(t,E(t_0))
\quad\text{for every $t\in I$.}
\]
It is well-known that this last identity is equivalent to the fact 
that the characteristic function $\rho(t,x):= 1_{E(t)}(x)$
is a weak solution of the transport equation \eqref{e:transport}
and, hence, of the continuity equation~\eqref{e:continuity}.
\end{parag}

%\begin{remark}
%When $u$ is divergence-free,
%equation \eqref{e:transport} agrees with
%the continuity equation~\eqref{e:continuity} and therefore
%the velocity field $u$ is compatible with the time-dependent
%domain $E$ if and only if the characteristic function
%$\rho(t,x):= 1_{E(t)}(x)$ is a weak solution
%of~\eqref{e:continuity}.
%\end{remark}

In the rest of this section, we address the following question:
given a time-dependent curve $\Gamma$ or a time-dependent domain $E$,
characterize under which conditions there exists a compatible, 
divergence-free velocity field $u$.

We begin with a general result, which we then specialize
according to our specific needs.
The proof of this result is postponed until the end of this section. 

\begin{proposition}
\label{s:basic}
Let $\Gamma$ be a time-dependent curve of class $C^k$, $k\ge 2$, in~$\R^2$
on the time interval~$I$, 
and let $\bar r:I\to(0,+\infty)$ be a continuous function.
Assume that, for every $t\in I$, the normal velocity
$v_\nn(t,\cdot)$ has compact support\,%
\footnoteb{This requirement is clearly redundant when $\Gamma$ is closed.}
and satisfies
\begin{equation}
\label{e:constvol}
\int_{\Gamma(t)} v_\nn(t,x) \, d\sigma(x) = 0
\, .
\end{equation}
Then there exists a divergence-free
velocity field $u:I\times\R^2\to\R^2$ of class
$C^{k-2}$ that is compatible with $\Gamma$
and such that the support of $u(t,\cdot)$
is contained in~${B(\Gamma(t),\bar r(t))}$
for every $t\in I$.

If, in addition, for
every $t\in I$ the support of $v_\nn(t,\cdot)$ is
contained in a compact, proper sub-arc $G(t)$ of $\Gamma(t)$,
which depends continuously in $t$,%
\footnoteb{Continuity is defined in terms of the Hausdorff distance between
compact subsets of $\R^2$.}
then $u$ can be chosen in such a way
that the support of $u(t,\cdot)$ is contained in~${B(G(t),\bar r(t))}$
for every $t\in I$.
\end{proposition}

\begin{remark}\label{s:rembasic}
\quad
%%%%%
\begin{enumerate}
[label=(\roman*),ref=(\roman*), 
leftmargin= 0 pt, itemsep=2pt, itemindent=30 pt]

\item 
%\label{s:rembasic.1} 
If $\Gamma$ is closed, Assumption \eqref{e:constvol}
is necessary, in the sense that it is satisfied by every
time-dependent closed path $\Gamma$ compatible
with a divergence-free velocity field $u$.
Indeed, for a fixed $t\in I$, we let $E(t)$
be the bounded open set with boundary $\Gamma(t)$,
and we denote by $\eta_{E(t)}$ the outer normal to $\bd E(t)$.
Then the divergence theorem yields
\[
\int_{\Gamma(t)} v_\nn \, d\sigma
=\pm\int_{\bd E(t)} u\cdot \eta_{E(t)} \, d\sigma
=\pm\int_{E(t)} \dive u \, dx
=0
\, ,
\]
where the sign $\pm$ depends on whether the normal
to $\Gamma(t)$ agrees with $\eta_{E(t)}$ or with~$-\eta_{E(t)}$.

\item 
%\label{s:rembasic.2} 
A modification of the previous argument shows that 
Assumption~\eqref{e:constvol} is necessary if $\Gamma$ is proper
and both $v_\nn$ and $u$ have compact support.
However, if we do not require that $u$ has compact support,
then we can drop both the assumption that $v_\nn$ has compact 
support and \eqref{e:constvol}.

\item 
\label{s:rembasic.3} 
If $\Gamma$ is a closed curve and agrees with the boundary 
of a bounded, time-dependent domain $E$, then it is well-known that
\[
\int_{\Gamma(t)} v_\nn \, d\sigma
= \frac{d}{dt} |E(t)|
\quad\text{for every $t\in I$.}
\]
Thus Assumption \eqref{e:constvol} is equivalent
to say that the area of $E(t)$ is constant in $t$.

\item
%\label{s:rembasic.4} 
Proposition~\ref{s:basic} can be generalized to higher dimensions, 
for instance to time-dependent surfaces with codimension one in $\R^n$, 
but such extensions require quite different proofs. 
\end{enumerate}
\end{remark}

We consider now the special case of a
curve that evolves homothetically in time.
We begin with a definition and a few remarks.

\begin{parag}[Homothetic curves]
\label{s:homcurve}
We say that a time-dependent curve
$\Gamma$ on the time interval $I$ is homothetic
in time if it can be represented as
\begin{equation}
\label{e:omot}
\Gamma(t)
= \lambda(t) \, \bar\Gamma
= \big\{ \lambda(t) \, x \, : \ x\in \bar\Gamma \big\},
\end{equation}
for some fixed curve $\bar\Gamma$ and
some function $\lambda:I\to(0,+\infty)$.

Let $\bar\gamma:J\to\R^2$ be a path that parametrizes $\bar\Gamma$.
Then the time-dependent path $\gamma:I\times J\to\R^2$ given by
\begin{equation}
\label{e:omot1}
\gamma(t,s):=\lambda(t) \, \bar\gamma(s)
\end{equation}
is a parametrization of $\Gamma$. 
Hence $\Gamma$ is of class $C^k$,
when $\bar\Gamma$ and $\lambda$ are of class~$C^k$.

Let $\bar\eta$ be the normal to $\bar\Gamma$ and let
$\bar v:\bar\Gamma\to\R$ be the function defined by
\begin{equation}
\label{e:omot1.1}
\bar v(x):= x\cdot \bar \eta(x)
\, .
\end{equation}
A simple computation starting from \eqref{e:omot1}
shows that the normal vector and the normal velocity
of $\Gamma$ (at $t\in I$ and $x\in\Gamma(t)$) are  given by
\begin{equation}
\label{e:omot2}
\eta(t,x)
= \bar\eta \big( x/\lambda(t) \big)
\, , \quad
v_\nn(t,x)
= \lambda'(t) \, \bar v \big( x/\lambda(t) \big)
\, .
\end{equation}

Finally, let $\bar u$ be any autonomous velocity field on $\R^2$
such that
\begin{equation}
\label{e:compat2}
\bar u(x) \cdot \bar \eta(x) = \bar v(x)
\end{equation}
for every $x\in \bar \Gamma$.
Then, using \eqref{e:omot2} one readily checks that
the time-dependent velocity field $u:I\times\R^2\to\R^2$ 
defined by
\begin{equation}
\label{e:omot3}
u(t,x):=\lambda'(t) \, \bar u \big( x/\lambda(t) \big)
\end{equation}
is compatible with the time-dependent curve $\Gamma$.
\end{parag}

The next result specializes the statement of Proposition~\ref{s:basic}
to the case of homothetic curves.

\begin{proposition}
\label{s:basicomot}
Let the function $\lambda: I\to(0,+\infty)$ and
the proper curve $\bar \Gamma$, both of class $C^k$, $k\ge 2$,
define a homothetic curve $\Gamma$ as in \eqref{e:omot}.
Let $\bar r$ denote a given positive number. 
Assume that there exists a compact sub-arc~$\bar G$ of $\bar\Gamma$
that contains the support of the function $\bar v$ defined
in \eqref{e:omot1.1}.
Assume, in addition, that
\begin{equation}
\label{e:constvol2}
\int_{\bar \Gamma} \bar v \, d\sigma = 0
\, .
\end{equation}
Then the following statements hold:
%%%%%
\begin{enumerate}
[label={\rm(\roman*)},ref=(\roman*),
itemsep=2pt, leftmargin=30 pt]

\item\label{s:basicomot.1}
there exists an autonomous velocity field
$\bar u$ on $\R^2$ of class $C^{k-2}$ which
satisfies~\eqref{e:compat2}, is divergence-free,
and its support is contained in $B(\bar G,\bar r)$;

\item\label{s:basicomot.2}
if $u$ is the time-dependent
velocity field defined in \eqref{e:omot3},
then $u$ is of class $C^{k-2}$, divergence-free, and compatible
with $\Gamma$, and the support of $u(t,\cdot)$
is contained in $B(\lambda(t)\, \bar G, \lambda(t)\, \bar r)$
for every $t\in I$.
\end{enumerate}
\end{proposition}

\begin{remark}\label{s:rembasic3}
\quad
%%%%%
\begin{enumerate}
[label=(\roman*),ref=(\roman*), 
leftmargin=0pt, itemsep=2pt, itemindent=30pt]

\item 
The formula for the normal velocity in \eqref{e:omot2}
shows that assumption~\eqref{e:constvol2} in
Proposition~\ref{s:basicomot} plays the role
of assumption \eqref{e:constvol} in Proposition~\ref{s:basic}.

\item 
Proposition~\ref{s:basicomot} does not apply to closed curves,
because condition \eqref{e:constvol2} is never verified if $\bar \Gamma$
is closed. Let indeed $E$ be the bounded open set with boundary~$\bar \Gamma$;
then the divergence theorem yields
\[
\int_{\bar\Gamma} \bar v \, d\sigma
= \int_{\bd E} x \cdot \bar \eta (x) \, d\sigma(x)
= \pm\int_E \dive(x) \, dx = \pm 2|E| \ne 0
\, ,
\]
where the sign $\pm$ depends on whether $\bar \eta$
is the inner or the outer normal of $E$.

\item
\label{s:rembasic3.3}
It is easy to check that the function $\bar v$ has compact support
if and only if the curve $\bar \Gamma$ agrees out of some  ball
$B=B(0,r)$ with two half-lines $L_-,L_+$ starting from the origin.
If in addition $\bar \Gamma$ is the boundary of an open set $E$
and we denote by $T$ the open set delimited by the half-lines
$L_-,L_+$ which agrees with $E$ outside $B$ (see Figure~\ref{f:commarea}),
then 
\[
\int_{\bar \Gamma} \bar v \, d\sigma
= \pm 2 \big( |E \setminus T| - |T \setminus E| \big)
\, .
\]
In particular assumption \eqref{e:constvol2}
is equivalent to saying that 
$E\setminus T$ and $T\setminus E$
have the same area.
\end{enumerate}
\end{remark}

\begin{figure}[ht]
\begin{center}
  \includegraphics[scale=1]{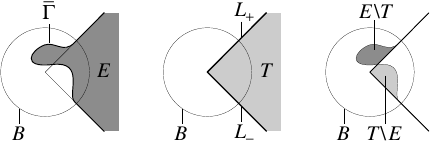}
  \caption{}%The situation described in Remark~\ref{s:rembasic3}\ref{s:rembasic3.3}.
  \label{f:commarea}
\end{center}
\end{figure}

The rest of this section is devoted to the proofs
of Propositions~\ref{s:basic} and \ref{s:basicomot}.
The key step is contained in Lemma~\ref{s:basiclemma} below.

\begin{parag}[Potential of a velocity field]
\label{s:potential}
Let $u:\R^2\to\R^2$ be a continuous velocity field
and let $\varphi:\R^2\to\R$ be a function of class $C^1$.
We say that $\varphi$ is a \emph{potential} for $u$ if
\[
u=\nabla^\perp \varphi
\, ,
\]
where $\nabla^\perp:=(-\bd_2,\bd_1)$.
Note that $u$ admits a potential if and only if it is divergence-free.
In the fluid dynamics literature, such $\varphi$ is called a stream
function for the flow generated by $u$.
\end{parag}

\begin{lemma}
\label{s:basiclemma}
Let  $\Gamma$ be a given curve, $v$ be a given function 
on $\Gamma$, both of class $C^k$ with $k\ge 2$, 
and $\bar r$ a positive number. Assume that the support 
of $v$ is contained in a compact (not necessarily proper) 
sub-arc~$G$ of $\Gamma$ and  that
\begin{equation}
\label{e:constvol3}
\int_\Gamma v \, d\sigma = 0
\, .
\end{equation}
Then there exists a divergence-free, autonomous velocity field
$u$ on $\R^2$ of class $C^{k-2}$, such that the normal component
of $u$ on $\Gamma$, that is, $u\cdot \eta$, agrees with $v$
and such that the support of $u$ is contained in $B(G,\bar r)$.
\end{lemma}

\begin{proof}
We describe the proof in the case $J=\R$ (recall that  $J$ is the
domain of the parametrization of the curve $\Gamma$); the case $J=\T^1$
requires few straightforward modifications.
In view of \S\ref{s:potential}, it suffices to find a potential
$\varphi:\R^2\to\R$ of class~$C^{k-1}$ with support contained
in $B(G,\bar r)$, such that
\begin{equation}
\label{e:compat4}
\bd_\tau\varphi = v
\quad\text{on $\Gamma$,}
\end{equation}
where $\tau$ is the tangent vector to $\Gamma$,
and then take $u:=-\nabla^\perp\varphi$.

Let $\gamma:\R\to\R^2$ be a parametrization of $\Gamma$.
For the construction of $\varphi$ we choose:
%%%%%
\begin{itemize}
[leftmargin=30pt, itemsep=2pt]

\item
a point $x_0=\gamma(s_0)\in\Gamma$ and,
if $G$ is a proper sub-arc of~$\Gamma$, we
further require that $x_0$ does not belong to $G$;
\item
a smooth function $g:\R\to\R$ with support
contained in $[-1/2,1/2]$ such that $g(0)=1$;
\item
a number $ r \in (0 , \bar r]$ strictly smaller
than the tubular radius of $\Gamma$.
\end{itemize}

\smallskip
Next, we consider the diffeomorphism $\Psi:\R\times(-r,r) \to B(\Gamma,r)$
defined in~\eqref{e:param}, and for every $x=\Psi(s,y)\in B(\Gamma,r)$
we set
\begin{equation}
\label{e:defpot}
\varphi(x)
=\varphi(\Psi(s,y))
:=  g(y/r) \, \int_{s_0}^s v(\gamma(s')) \, |\dot\gamma(s')| \, ds'
\, .
\end{equation}

\smallskip
If $x$ belongs to $\Gamma$, then $x=\gamma(s)=\Psi(s,0)$. 
Therefore, $\varphi(x)$ is the integral of~$v$ along the (oriented)
sub-arc of $\Gamma$ starting from $x_0$ and ending at $x$, so that 
the restriction of $\varphi$ to $\Gamma$ is a primitive of $v$
and  satisfies \eqref{e:compat4}.

Now, Formula~\eqref{e:defpot} shows that $\varphi\circ\Psi$
is a function of class $C^k$ on $\R\times(-r,r)$ with support 
contained in $\R\times[-r/2,r/2]$. 
Since $\Psi$ is a diffeomorphism of class $C^{k-1}$ and maps
$\R\times[-r/2,r/2]$ into the closure of $B(\Gamma,r/2)$, we deduce
that~$\varphi$ is a function of class $C^{k-1}$ on $B(\Gamma,r)$
with support contained in the closure of~$B(\Gamma,r/2)$.
We complete the construction extending $\varphi$ by $0$ to the 
complement of this neighborhood in~$\R^2$.

\smallskip
It remains to check that the support of $\varphi$
is contained in $B(G,\bar r)$. When~${G=\Gamma}$,
this follows from the fact that the support of $\varphi$
is contained in the closure of~$B(\Gamma,r/2)$, which in turn 
is contained in $B(\Gamma, \bar r)$. When $G=\gamma([s_1,s_2])$ 
is instead a proper sub-arc of $\Gamma$, we have that:
%%%%%
\begin{itemize}
[leftmargin=30pt, itemsep=2pt]

\item
$v(\gamma(s))=0$ for $s\notin [s_1,s_2]$ by assumption;
\item
$s_0\notin [s_1,s_2]$ by the choice of $x_0$;
\item
Condition \eqref{e:constvol3} can be re-written as
$\int_{s_1}^{s_2} v(\gamma(s')) \, |\dot\gamma(s')|\, ds'=0$.
\end{itemize}
Putting together these facts and recalling the choice of $g$,
one easily shows that~${\varphi(\Psi(s,y))=0}$,
if $s\notin[s_1,s_2]$ or $y\notin[-r/2,r/2]$,
and then
\[
\mathrm{supp}(\varphi)
\subset \Psi \big( [s_1,s_2]\times[-r/2,r/2] \big)
\subset B(G,\bar r)
\, .
\qedhere
\]
\end{proof}

\smallskip

\begin{proof}[Proof of Proposition~\ref{s:basicomot}]
Statement~\ref{s:basicomot.1} follows from Lemma~\ref{s:basiclemma},
while statement~\ref{s:basicomot.2} is an immediate consequence 
of~\ref{s:basicomot.1} and~\S\ref{s:homcurve}.
\end{proof}

\begin{proof}[Proof of Proposition~\ref{s:basic}]
For every $t\in I$, we use Lemma~\ref{s:basiclemma}
to construct a divergence-free velocity field $u(t,\cdot)$
of class $C^{k-2}$, which satisfies the compatibility condition
\eqref{e:compat} at time $t$, and the support of which is contained
in $B(G(t),\bar r(t))$.

However, this construction gives only that $u$ is of class $C^{k-2}$
in the variable $x$.
To show that $u$ can be taken of class $C^{k-2}$ in $t$ and $x$,
we re-examine the proof of Lemma~\ref{s:basiclemma}.
The key point in that proof is the
regularity of class $C^{k-1}$ in the variables $t,s,y$
of the right-hand side of formula \eqref{e:defpot},
which in our specific case is given by 
\[
g( y/r(t) ) \, \int_{s_0(t)}^s
v_\nn(t,\gamma(t,s')) \, |\dot\gamma(t,s')| \, ds'
\, .
\]
It is clear that this expression has the required regularity
provided that we choose~$r(t)$ and $s_0(t)$ at least of class $C^{k-1}$
in $t$.

Since both $\bar r(t)$ and the tubular radius of $\Gamma(t)$
are continuous, strictly positive functions of $t$,
it is always possible to choose $r(t)$ smaller than both,
strictly positive, and smooth in $t$.

If we only require that the support of $u$ is contained
in $B(\Gamma(t),\bar r(t))$, we can take~$s_0(t)$ constant in $t$.
If we require that the support of $u$ is contained in~${B(G(t),\bar r(t))}$, 
then we can again choose $s_0(t)$ smooth in $t$, but the existence of such 
a choice is more delicate, and relies on the fact that $G(t)$ is
a proper sub-arc for all $t\in I$.
\end{proof}

%
%	SECTION 5
%

\section{First example: pinching}
\label{pinching}

In this section we verify Assumption~\ref{ass} for $s=1$ and for 
every $1 \leq p<\infty$.
In~fact, we obtain slightly more. 
We construct a velocity field  $u_0$ and a solution~$\rho_0$ of the 
continuity equation \eqref{e:continuity} on~$\R^2$, both compactly 
supported in the open unit square $\Qone$, where the velocity has Sobolev 
regularity $W^{s,p}$, and we do so for each $s$ and $p$ such that
$W^{s,p}$ does not embed continuously in the Lipschitz class, that~is,
\begin{equation}
\label{e:range}
\text{
$s<1$ and $p\le \infty$,
or $s\ge 1$ and $\displaystyle p < \frac{2}{s-1}$.
}
\end{equation}
This first construction exploits topological changes to the evolution 
of a certain sets and, therefore, cannot be realized with a Lipschitz 
velocity field.

\smallskip
More precisely, we give an example of $u_0$ and $\rho_0$,
both defined for $0 \le t \le 1$, such that:
%%%%%
\begin{enumerate}
[label=(\alph*),
leftmargin=30pt, itemsep=2pt]

\item 
$u_0$ is a time-dependent, bounded, and divergence-free velocity 
field on $\R^2$, which is compactly supported on the open unit 
square $\Qone$.
The field $u_0$ is smooth in both variables $t$ and $x$ for $t\ne k/8$, 
$k=1,\dots,7$, and bounded in $\dot{W}^{s,p}(\R^2)$
uniformly in $t$ for~$s$ and~$p$ in the range~\eqref{e:range};

\item 
$\rho_0$ is of the form $\rho_0(t,\cdot) = 1_{E(t)} - \pi/16$,
where $E(t)$ is a time-dependent domain in $\R^2$ defined for 
$0\le t \le 1$, with the property that its closure is contained
in $\Qone$ and its area equals $\pi/16$ (thus~$\rho_0(t,\cdot)$ 
has average zero).
The set $E(t)$ is continuous in $t$,%
\footnoteb{Again, continuity is defined in terms of the Hausdorff 
distance between compact subsets.}
and smooth for $t\ne k/8$, $k=1,\dots,7$;

\item 
$E(0)$ is the disk with center $0$ and radius $1/4$,
while $E(1)$ is the union of the four disks with
centers $(\pm 1/4,\pm 1/4)$ and radius $1/8$.
\end{enumerate}
Since  $u_0$ and $\rho_0$ have compact support in the open square 
$\Qone$, we can canonically identify them with fields and functions 
defined on the torus~$\T^2$. 
Remark~\ref{r:tay}\ref{r:tay:4} ensures that $u_0$ is then bounded 
in $\dot{W}^{s,p}(\T^2)$ for the same $s$ and $p$.

Therefore, Assumption~\ref{ass} is satisfied for $s$ and $p$ in the 
range~\eqref{e:range}, in particular for $s=1$ and for 
every~$1 \leq p<\infty$.

\medskip
Specifically, we construct a time-dependent domain $E$, satisfying 
the conditions listed above,  and a velocity field $u_0$ defined for
$t\ne k/8$, $k=1,\dots,7$, which is smooth and compatible with $E$.
As a consequence,  the characteristic function $1_{E(t)}$ is a weak 
solution of the continuity equation \eqref{e:continuity} in the open
time intervals $((k-1)/8,k/8)$, $k=1,\dots,8$.
The fact that it is also a solution on the time interval $[0,1]$
is ensured by the continuity in $t$.
The set $E(t)$ for $t=k/8$ with $k=0,\dots,4$ and $t=1$
is described in Figure~\ref{f:evol1}.

\begin{figure}[ht]
\begin{center}
  \includegraphics[scale=1.2]{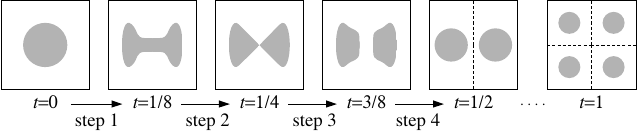}
  \caption{The set $E(t)$ for $t=k/8$ with $k=0,\dots,4$ and $t=1$.}
  \label{f:evol1}
\end{center}
\end{figure}

To describe this construction in more details, we denote by $B$
the open disk with center $0$ and radius $1/4$, and by $T$ the cone
in $\R^2$ such that $|x_2| < |x_1|$.
Next, for $t=1/8$, $t=1/4$, and $t=3/8$ we choose a smooth set $E(t)$
shaped as in~Figure~\ref{f:evol1} making sure that
%%%%%
\begin{enumerate}
[label=(\alph*), ref=(\alph*), start=4,
itemsep=2pt, leftmargin=30pt]

\item
$E(t)$ is symmetric with respect to both axes;
\item
$E(t)$ has area $\pi/16$;
\item
$E(t)\setminus B$ is the same set at the times $t$ chosen above 
(i.e., $t=1/8$, $t=1/4$, and $t=3/8$);
\item
\label{first.7}
$(E(1/8)\setminus T) \cap B$ and $(T\setminus E(1/8))\cap B$ 
have the same area. 
\end{enumerate}

\medskip
In the rest of this section we describe the construction
of $E(t)$ and $u_0(t,\cdot)$ for $t$
in the time intervals $[0,1/8]$ (Step~1 in Figure~\ref{f:evol1})
and $(1/8,1/4)$ (Step~2 in Figure~\ref{f:evol1}).
The construction in the remaining time intervals (steps)
is similar, and is omitted.

\medskip
\emph{Step~1: construction of $E(t)$ and $u_0(t,\cdot)$ 
for $0 \le t \le 1/8$.}
Since the sets $E(0)$ and $E(1/8)$ are both smooth and have area 
$\pi/16$, we can clearly find a time-dependent $E(t)$ for $0<t<1/8$ 
that deforms $E(0)$ to $E(1/8)$ such that $E(t)$ has constant area 
$\pi/16$ and such that the map $t\mapsto E(t)$ is smooth on $[0,1/8]$. 
Then, by Proposition~\ref{s:basic} we can find a smooth velocity 
field $u_0:[0,1/8]\times\R^2\to\R^2$ that is divergence-free and 
compatible with $E$.
Moreover, since $\bd E(t)$ is contained in~$\Qone$, we can assume 
that the support of $u_0$ is contained in $\Qone$ for all $t$.
In particular all positive Sobolev norms of $u_0(t,\cdot)$ are 
uniformly bounded in $t$.

\begin{figure}[h]
\begin{center}
  \includegraphics[scale=1]{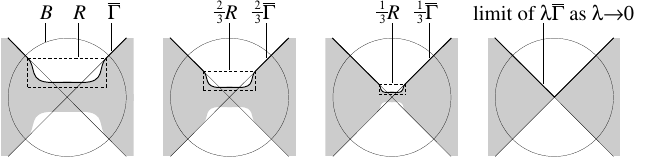}
  \caption{The curve $\ove{\Gamma}$, the homothetic copies $\lambda\,\ove{\Gamma}$
  with $\lambda=2/3$, $\lambda=1/3$, and their limit as $\lambda\to 0$.
  The circle is centered at $0$ and has radius $1/4$. The set $E(t)$ is in gray.}
  \label{f:evol3}
\end{center}
\end{figure}

\emph{Step~2: construction of $E(t)$ and $u_0(t,\cdot)$ for $1/8 < t < 1/4$.}
Let $\ove{\Gamma}$ be the proper curve drawn in Figure~\ref{f:evol3}.
More precisely, $\ove{\Gamma}$ is defined outside $B$ by the equation
$x_2=|x_1|$, and agrees in $B$ with the connected component
of the boundary $\bd E(1/8)\cap B$ that lies in the upper half plane.

We pick a smooth  decreasing function
$\lambda$ on $[1/8,1/4)$ such that $\lambda(1/8)=1$
and $\lambda(t)$ tends to $0$ as $t\to 1/4$. 
The function $\lambda$ will be explicitly defined later
in order to satisfy further requirements.

Then we select the sets $E(t)$, $1/8<t<1/4$, 
satisfying the following requirements in addition to 
preserving area and smoothness in time:
%%%%%
\begin{enumerate}
[label=(\alph*), ref=(\alph*), start=8,
itemsep=2pt, leftmargin=30pt]

\item
$E(t)$ agrees with $E(1/8)$ outside $B$;

\item
\label{first.9}
$\bd E(t) \cap B$ has two connected components, which are symmetric
with respect to both axes, and the component that lies in the upper 
half plane agrees with $\lambda(t) \, \ove{\Gamma}$ in $B$ (the set 
$E(t)$ is drawn in gray in Figure~\ref{f:evol3} for $\lambda(t)=1$, 
$\lambda(t)=2/3$, and $\lambda(t)=1/3$).
\end{enumerate}

By Remark~\ref{s:rembasic3}\ref{s:rembasic3.3}, Property~\ref{first.7} 
above implies that the curve $\ove{\Gamma}$ satisfies \eqref{e:constvol2}
and, therefore, we can apply Proposition~\ref{s:basicomot}\ref{s:basicomot.2},
to obtain a smooth, divergence-free velocity field $w:[1/8,1/4)\times\R^2\to\R^2$ 
that is compatible with the homothetic curve $\lambda(t)\,\ove{\Gamma}$.
Moreover, $w(t,\cdot)$ is compactly supported in the upper half-plane 
$\R\times(0,+\infty)$ for all $t$ (specifically, we can require that the 
support is contained in the dashed rectangle $R$ in Figure~\ref{f:evol3}).

Finally, we take $u_0: [1/8,1/4)\times\R^2\to\R^2$ equal to $w$ in the 
upper half-plane, and we extend it to the lower half-plane by reflection.
In this way,  $u_0$ is still smooth and compactly supported, and by 
Property~\ref{first.9} above it is compatible with the time-dependent
domain $E$ inside the ball $B$.
On the other hand, $u_0$ vanishes outside the ball $B$ and, therefore, 
is compatible with the set $E \setminus B$, which is constant in time 
(recall again Property~\ref{first.7}).
In conclusion, $u_0$ is compatible with~$E$.

It remains to choose $\lambda$ so that $u_0(t,\cdot)$ is bounded in
$\dot{W}^{s,p}(\R^2)$ uniformly in $t\in [1/8,1/4)$ for every $s,p$
as in \eqref{e:range}.To this end, we recall that by 
Proposition~\ref{s:basicomot}\ref{s:basicomot.2},
the field $u_0$ can be written in the form
\[
u_0(t,x) = \lambda'(t) \, \bar u \big( x/\lambda(t) \big),
\]
where $\bar u:\R^2\to\R^2$ is smooth and compactly supported.
Therefore, using~\eqref{e:scaleR}, for every $t$ we have 
\begin{align*}
\| u_0(t,\cdot) \|_{\dot{W}^{s,p}(\R^2)}
= |\lambda'(t)| \, |\lambda(t)|^{2/p-s} \| \bar u \|_{\dot{W}^{s,p}(\R^2)}
\, .
\end{align*}
Now, a simple computation shows that $u_0(t,\cdot)$ is
bounded in $\dot{W}^{s,p}(\R^2)$ uniformly in time
for all~$s$ and~$p$ as in~\eqref{e:range} if we take
\[
\lambda(t) := \exp \Big( 2- \frac{1}{1-4t} \Big)
\, .
\]
In particular $u_0$ is a bounded function in both space and time.

\begin{remark}%\label{r:messrlf}
\quad
%%%%%
\begin{enumerate}
[label=(\roman*),ref=(\roman*), 
leftmargin=0pt, itemsep=2pt, itemindent=30pt]

\item 
The flow of the (non-Lipschitz) velocity field $u_0$ changes the topology
of sets: the ball at time $t=0$ is transformed into two balls at time~${t=1/2}$.

\item
\label{r:messrlf.2} 
Using symmetry considerations, it is possible to check that the flow 
of $u_0$ compresses a vertical segment to a point, namely the origin 
(the center of the circle in Figure~\ref{f:evol3}), from time $t=1/8$
to time $t=1/4$.
Similarly,  the flow of $u_0$ expands a point, the origin, to a horizontal 
segment from time $t=1/4$ to time $t=3/8$. 

\item 
In particular, non-uniqueness holds for the characteristic curves 
of the velocity field $u_0$ starting at any point laying on the 
vertical segment referenced in point~\ref{r:messrlf.2} above at time $t=1/8$.
\end{enumerate}
\end{remark}

%
%	SECTION 6
%

\section{Scaling analysis in a quasi-self-similar construction}
\label{quasiselfsimilar}

As already noted, any velocity field with properties similar to those 
of the field constructed in Section~\ref{pinching} cannot have Lipschitz 
regularity, since sets evolved in time by the associated flow
do not preserve their connectivity. 
Indeed, it is not evident how to build an example that
satisfies Assumption~\ref{ass} in the case when $s=1$ and $p=\infty$.
In this section, we address this case by replacing the
(exactly) self-similar scheme of Section~\ref{selfsimilar} with a
{\em quasi-self-similar} scheme.
That is, instead of replicating rescaled copies of
{\em one} basic element at each step of the evolution
(as in~\S\ref{ss:ssc}), we consider a {\em finite family}
of basic elements, which are rescaled and rearranged at each step
of the evolution according to a certain combinatorial pattern.

For the quasi-self-similar scheme, we work on the full plane~$\R^2$. 
We implement this construction and produce a concrete example 
in~Section~\ref{snake}.

\medskip

Given $s > 0$ and $1 \leq p \leq \infty$, Assumption~\ref{ass} is
replaced by the assumption below, where we denote
by $\lceil s \rceil$ the smallest integer greater or equal than $s$.
We recall that $\dot{W}^{s,\infty}$ is the Lipschitz-H\"older 
space defined in~\S\ref{ss:lhs}.

\begin{parag}[Assumption:~basic family]
\label{ass2}
There exists an integer $N$ such that, for $j = 1,\ldots,N$,
there are velocity fields~$u_j$ and corresponding (not identically zero)
solutions $\rho_j$ to~\eqref{e:continuity}, all defined for
$0 \leq t \leq 1$ and $x \in \R^2$, satisfying:
%%%%%
\begin{enumerate}
[label=(\roman*),ref=(\roman*),
itemsep=2pt, leftmargin=30pt]

\item\label{ass2.1} 
each velocity field $u_j$ is bounded, divergence-free, tangent
to the boundary of the square~$\Qone$,%
\footnoteb{We observe that the normal trace of $u_j$ on 
$\bd\Qone$ (from the interior as well as from the exterior
of the set) is well defined in distributional sense,
because $u_j$ is divergence-free.}
and bounded in $\dot{W}^{\lceil s \rceil,p}(\R^2)$ uniformly in time;

\item\label{ass2.2} 
each solution $\rho_j$ is a bounded function and has zero average 
on $\Qone$ for all times;

\item \label{ass2.3} 
there exists a positive constant $\lambda$, with $1/\lambda$ an
integer greater or equal than~$2$, such that each function $\rho_j(1,\cdot)$
agrees on each square of the tiling~$\Til_\lambda$ (introduced in 
Definition~\ref{d:tiling}) with one of the functions $\rho_i(0,\cdot)$ with 
$1\leq i\leq N$ after rescaling and a possible translation,
that is,  for each $Q \in \Til_\lambda$ and for all $x\in Q$,
\[
\rho_j (1,x) =  \rho_{i(j,Q)} \left( 0 , \frac{x-r_Q}{\lambda} \right)
\, ,
\]
for a suitable $i = i(j,Q) \in \{1,\ldots,N\}$ and for some point $r_Q=r_Q(j)$. 
\end{enumerate}
We remark that we \emph{do not} assume that the supports of $u_j$ and 
$\rho_j$ are contained in the closure of~$\Qone$.
\end{parag}

\begin{parag}[Quasi-self-similar construction]\label{ss:ass}
Under Assumption~\ref{ass2}, we now define inductively a quasi-self-similar 
scheme that will be used to give our second, Lipschitz-continuous, 
example of optimal mixer. 

\smallskip
\emph{Initial step.} We start by choosing
a positive constant $\bar \lambda$, with $1 / \bar\lambda$
an integer greater or equal than~$1$.%
\footnoteb{We allow $\bar \lambda=1$ here, but in the example 
in Section~\ref{snake} we will take $\bar \lambda=2$.}
We define the evolution for
$0 \leq t \leq 1$ by patching together velocity fields and
solutions on the tiling $\Til_{\bar \lambda}$ of $\Qone$.

For every $Q \in \Til_{\bar \lambda}$, we select an index 
$\bar \jmath (Q) \in \{1,\ldots,N\}$
and we set for $x \in Q$ and $0 \leq t \leq 1$:
\begin{equation}\label{e:gluefieldbase}
u(t,x) 
:= \bar \lambda \, u_{\bar \jmath (Q)} \left( t , \frac{x-r_Q}{\bar \lambda} \right)  
\,, \quad
\rho(t,x) 
:= \rho_{\bar \jmath (Q)} \left( t , \frac{x-r_Q}{\bar \lambda} \right) 
\,.
\end{equation}
For $x \not \in \Qone$, we set both $u$ and $\rho$ equal to zero.

We stress that, in this step (as well as  in the iterative step below), 
the resulting field is divergence-free, but it does not necessarily have 
Sobolev regularity, since the derivative may jump at the boundary 
of the patch. 
In what follows, we will temporarily {\em assume} the needed regularity 
(see Assumption~\ref{ass3} below), and show afterwards that it is, in 
fact, fulfilled for the specific example in~Section~\ref{snake}. 

Since by construction the velocity field $u$ in \eqref{e:gluefieldbase}
is tangent to the boundary of all the tiles in~$\Til_{\bar \lambda}$, 
it follows that, for $0<t\leq 1$, the function~$\rho$ in \eqref{e:gluefieldbase}
is a weak solution of  the continuity equation with velocity field $u$ globally
in $\R^2$.
We also note that, by Assumption~\ref{ass2}\ref{ass2.3}, the solution at 
time $1$, $\rho(1,\cdot)$, agrees on each element of the tiling 
$\Til_{\bar \lambda}$ with one of the functions $\rho_i(0,\cdot)$ after
rescaling and possible translation.

\smallskip
\emph{Iterative step.}
For a given positive parameter $\tau$ (to be chosen later), we define
\[
T_n := \sum_{i=0}^{n-1} \tau^i  
\quad\text{for $n=1, 2, \ldots, \infty$.}
\]
We next inductively assume that $u$ and $\rho$ have been defined for 
$0 \leq t \leq T_n$, in such a way that on each square of the tiling 
$\Til_{\bar \lambda \lambda^n}$ the function $\rho(T_n,\cdot)$ agrees 
with a rescaled translation of one of the functions $\rho_i(0,\cdot)$. 
We then show how to define~$u$ and~$\rho$ for $T_n < t \leq T_{n+1}$.

We consider a square $Q \in \Til_{\bar \lambda \lambda^n}$. 
By the inductive assumption, there exists an index $j = j(n,Q)$ such that
\begin{equation}
\label{e:choose}
\rho(T_n,x) 
= \rho_j \left(0 , \frac{x-r_Q}{\bar \lambda \lambda^n} \right) 
\quad\text{for $x \in Q$.}
\end{equation}
Accordingly, for $x \in Q$ and $T_n< t \leq T_{n+1}$ we define 
\begin{equation}
\label{e:gluefield}
u(t,x) 
:= \frac{\bar \lambda \lambda^n}{\tau^n} u_j 
   \left( \frac{t-T_n}{\tau^n} , \frac{x-r_Q}{\bar \lambda \lambda^n} \right)  
\,, \quad
\rho(t,x) 
:= \rho_j \left( \frac{t-T_n}{\tau^n} , \frac{x-r_Q}{\bar \lambda \lambda^n} \right)
\,.
\end{equation}
As before, for $x \not \in \Qone$ we set both $u$ and $\rho$ equal to zero.
By the same argument as in the initial step, we have that,
for $T_n< t \leq T_{n+1}$, the function
$\rho$ in \eqref{e:gluefield} is a weak solution of the continuity equation
with velocity field $u$ globally in $\R^2$.

Again by Assumption~\ref{ass2}\ref{ass2.3}, on each square of the tiling
$\Til_{\bar \lambda \lambda^{n+1}}$ the function~$\rho(T_{n+1},\cdot)$
agrees with a rescaled translation of one of the functions
$\rho_i(0,\cdot)$. This concludes the inductive
procedure, which gives a velocity field $u$ and a weak solution
$\rho$ of \eqref{e:continuity} defined for a.e.~$x \in \R^2$
and for all $0 \leq t < T_\infty$.
\end{parag}

We now make a further assumption on the velocity field $u$ obtained by 
the quasi-self-similar scheme that we have described. 
One drawback of our construction is, in fact, 
that we do not {\em a priori} control  the behavior of derivatives of the field at 
the boundary of each patch. We are therefore forced at this stage to make a further 
assumption on $u$, concerning its regularity. 

\begin{parag}[Assumption: regularity of the patching]\label{ass3}
The velocity field $u(t,\cdot)$ belongs to 
$\dot{W}^{\lceil s \rceil ,p}(\R^2)$ for all $0 \leq t < T_\infty$.
\end{parag}

A few remarks on this delicate point are in order.

\begin{remark}
\quad
%%%%%
\begin{enumerate}
[label=(\roman*), 
leftmargin=0pt, itemsep=2pt, itemindent=30pt]

\item 
The fact that Assumptions~\ref{ass2} and \ref{ass3} 
entail regularity of order $\lceil s \rceil$ (rather than $s$)
is technical and due to the fact that the norm in a Sobolev space
with integer order is local, a property that we will exploit in 
the proof of Lemma~\ref{l:fieldqss}.

\item 
Assumption~\ref{ass3} is, in fact, the key structural condition 
to ensure that  a quasi-self-similar construction yields velocity 
fields with the required regularity, as already observed above. 
It is indeed easy  to construct families of velocity fields and 
solutions that satisfy Assumption~\ref{ass2}, but not 
Assumption~\ref{ass3}. 

\item 
In the relevant case $s=1$ and $p=\infty$, i.e., in the Lipschitz case,
Assumption~\ref{ass3} is equivalent to assume that $u$ has a continuous
representative on~$\R^2$. In fact, it is sufficient to assume the continuity
of $u$ across the boundary of adjacent squares in each tiling.
\end{enumerate}
\end{remark}

We stress that in Assumption~\ref{ass3},  we do not require
the Sobolev norm of $u$ to be bounded uniformly  with respect to time. 
The uniformity of the Sobolev bounds in time is then guaranteed by the
following lemma.

\begin{lemma}\label{l:fieldqss} Let $\tau = \lambda^{1-s}$. 
Under Assumptions~\ref{ass2} and \ref{ass3}, the velocity field~$u$ 
constructed by the quasi-self-similar procedure in~\S\ref{ss:ass} 
is divergence-free and is bounded in $\dot{W}^{s,p}(\R^2)$ uniformly 
in time.
\end{lemma}

\begin{proof}
First of all, since each velocity field $u_j$ is divergence-free
in $\Qone$ and tangent to the boundary $\bd\Qone$, it follows
that $u$ is globally divergence-free. It remains to prove 
the bound on the $\dot{W}^{s,p}(\R^2)$ norm. 

\emph{Step~1: the case $s=k$ an integer.}
Let $T_n < t \leq T_{n+1}$ for some $n\in \NN$. 
Then $u(t,\cdot)$ is defined as in \eqref{e:gluefield} 
for some function $j = j(n,Q)$. 
Assumption~\ref{ass3} guarantees that the $\dot{W}^{k,p}(\R^2)$ 
norm of $u(t,\cdot)$ is finite, therefore we only need to estimate 
the sum of the $\dot{W}^{k,p}$ norms of the restriction of $u(t,\cdot)$
to the squares $Q$ in $\Til_{\bar \lambda \lambda^n}$: 
\begin{align*}
\| u(t,\cdot) \|_{\dot{W}^{k,p}(\R^2)}^p
& =
\int_{\Qone} |\nabla^k u(t,x)|^p \, dx \\
& =
\sum_{Q \in \Til_{\bar \lambda \lambda^n}} 
\int_{Q} \left| \nabla^k
\left( \frac{\bar \lambda \lambda^n}{\tau^n} \, u_j 
     \left( \frac{t-T_n}{\tau^n} , \frac{x-r_Q}{\bar \lambda \lambda^n} \right) 
\right) \right|^p \, dx \displaybreak[3] \\
& =
\sum_{Q \in \Til_{\bar \lambda \lambda^n}} 
\int_{Q} \left| \frac{\bar \lambda \lambda^n}{\tau^n \bar\lambda^k \lambda^{kn}} 
\left( \nabla^k u_j \right) \left( \frac{t-T_n}{\tau^n} , \frac{x-r_Q}{\bar \lambda \lambda^n} \right)  
\right|^p \, dx \displaybreak[3] \\
&=
\left( \frac{\lambda^{1-k}}{\tau} \right)^{pn} \bar \lambda^{p(1-k)} 
\sum_{Q \in \Til_{\bar \lambda \lambda^n}} 
\int_\Qone \left| \left( \nabla^k u_j \right) \left( \frac{t-T_n}{\tau^n} , y \right) \right|^p  
(\bar \lambda\lambda^{n})^2 \, dy \\
& \leq
\left( \frac{\lambda^{1-k}}{\tau} \right)^{pn} 
\bar \lambda^{p(1-k)} \max_{1 \leq j \leq N} \sup_{0 \leq r \leq 1} 
\| u_j(r,\cdot) \|_{\dot{W}^{k,p}(\R^2)}^p \,.
\end{align*}
The computation for $p=\infty$ is similar and gives
\[
\| u(t,\cdot) \|_{\dot{W}^{k,\infty}(\R^2)}
\leq
\left( \frac{\lambda^{1-k}}{\tau} \right)^n \bar \lambda^{(1-k)} 
\max_{1 \leq j \leq N} \sup_{0 \leq r \leq 1} \| u_j(r,\cdot) \|_{\dot{W}^{k,\infty}(\R^2)} 
\,.
\]
The fact that $\tau=\lambda^{1-s}$ gives the desired bound and 
concludes the proof for $s=k$ integer.

\emph{Step~2: the general case $s \geq 0$ and real.} 
We rely on the previous step and we use~\eqref{e:interpW} with
$s_1 = 0$, $s_2 = \lceil s \rceil$ and 
$\vartheta = 1- s / \lceil s \rceil$, obtaining
\begin{align*}
\| u(t,\cdot) \|_{\dot{W}^{s,p}(\R^2)}
& \leq
\| u(t,\cdot) \|^{\vartheta}_{L^p(\R^2)} \,
\| u(t,\cdot) \|^{1-\vartheta}_{\dot{W}^{\lceil s \rceil,p}(\R^2)} \\
& \leq \left( \frac{\lambda}{\tau} \right)^{\vartheta n} \bar \lambda^{\vartheta}
\left( \frac{\lambda^{1-\lceil s \rceil}}{\tau} \right)^{(1-\vartheta) n} \bar \lambda^{(1-\lceil s \rceil)(1-\vartheta)} \, M_{s,p} \\
& = \left( \frac{\lambda^{1-s}}{\tau} \right)^n \bar \lambda^{1-s} \, M_{s,p} \,,
\end{align*}
where
\[
M_{s,p} 
:= \max_{1 \leq j \leq N} \sup_{0 \leq t \leq 1} \left[ \| u_j(t,\cdot) \|_{L^p(\R^2)}
   + \| u_j(t,\cdot) \|_{\dot{W}^{\lceil s \rceil ,p}(\R^2)} \right] 
\,,
\]
Above we have used the estimate in Step 1 for  $k=0$ and $k = \lceil s \rceil$.

Again, the choice $\tau = \lambda^{1-s}$ allows to conclude.
\end{proof}

\begin{parag}[Decay of the functional mixing scale]
We now analyze the behavior in time of negative Sobolev norms 
of the solution $\rho$ constructed in~\S\ref{ss:ass}. 
For~${T_n \leq t < T_{n+1}}$ we have
\[
\rho(t,x) 
= \sum_{Q \in \Til_{\bar \lambda \lambda^n}}
\rho_j \left( \frac{t-T_n}{\tau^n} , \frac{x-r_Q}{\bar \lambda\lambda^n} \right) \, 1_Q(x)
\quad \text{for $x \in \R^2$,}
\]
for a suitable $j = j(n,Q)$. For any $r>0$, Equation~\eqref{e:scaleR} implies
that
\begin{align}
\| \rho(t,\cdot)  \|_{\dot{H}^{-r}(\mathbb{R}^2)}
& \leq
\sum_{Q \in \Til_{\bar \lambda \lambda^n}} \left\| \rho_j \left( \frac{t-T_n}{\tau^n} , \frac{x-r_Q}{\bar \lambda \lambda^n} \right) \, 1_Q(x) \right\|_{\dot{H}^{-r}(\mathbb{R}^2)} \notag \\
& =
\sum_{Q \in \Til_{\bar \lambda \lambda^n}} \bar \lambda^{1+r} \lambda^{n(1+r)} \left\| \rho_j \left( \frac{t-T_n}{\tau^n} , \cdot \right) \, 1_\Qone \right\|_{\dot{H}^{-r}(\mathbb{R}^2)} \notag \\
& \leq  \frac{\bar \lambda^{1+r} \lambda^{n(1+r)}}{\bar \lambda^2 \lambda^{2n}} M_r = \bar \lambda^{r-1} \big(\lambda^{r-1}\big)^n M_r \,, \label{e:r}
\end{align}
with $M_r$ defined (for all $r \geq 0$) as
\begin{equation}\label{e:Msigma}
M_r := \max_{j=1,\ldots,N} \sup_{0 \leq t \leq 1}  
       \| \rho_j(t,\cdot) \, 1_\Qone \|_{\dot{H}^{-r}(\mathbb{R}^2)} 
\,.
\end{equation}
Since each $\rho_j$ is bounded and has zero average on $\Qone$, 
Remark~\ref{r:fctsinH}\ref{r:fctsinH:3} implies that~$M_r$ is 
finite for $0 \leq r < 2$.

Estimate~\eqref{e:r} gives the correct decay of the 
homogeneous norms $\dot{H}^{-r}(\mathbb{R}^2)$ for all 
$1 < r <2$, since in this case $\lambda^{r-1} < 1$ 
and $M_r < \infty$. In order to prove the 
decay of the homogeneous norm $\dot{H}^{-1}(\mathbb{R}^2)$
we need an interpolation argument. 
Using~\eqref{e:interpH} we find that,
for $1<r<2$,
\[
\| \rho (t,\cdot)  \|_{\dot{H}^{-1}(\mathbb{R}^2)}
\leq \| \rho (t,\cdot)  \|^{1/r}_{\dot{H}^{-r}(\mathbb{R}^2)}
     \| \rho (t,\cdot)  \|^{1-1/r}_{L^2(\mathbb{R}^2)}
\,.
\]
Together with~\eqref{e:r}, this estimates gives,
for $T_n \leq t < T_{n+1}$,
\begin{equation}\label{e:cg}
\| \rho(t,\cdot)  \|_{\dot{H}^{-1}(\mathbb{R}^2)}
\leq \bar \lambda^{1-1/r}  M_0^{1-1/r} M_{r}^{1/r}  \big(\lambda^{1 - 1/r}\big)^n \,.
\end{equation}
Setting
\[
c_r := 1-1/r > 0 
\,, \quad
C_r := \bar \lambda^{1-1/r} M_0^{1-1/r} M_{r}^{1/r} > 0 
\,,
\]
we obtain from~\eqref{e:cg} that, for $T_n \leq t < T_{n+1}$,
\begin{equation}
\label{e:decaysigma}
\| \rho(t,\cdot) \|_{\dot{H}^{-1}(\mathbb{R}^2)}
\leq C_r \big(\lambda^{c_r}\big)^n 
\,.
\end{equation}
In particular, choosing $r=3/2$ yields
\begin{equation}
\label{e:decay2}
\| \rho(t,\cdot) \|_{\dot{H}^{-1}(\mathbb{R}^2)}
\leq \bar \lambda^{1/3}
M^{1/3}_0 M^{2/3}_{3/2} \; \big(\lambda^{1/3}\big)^n 
\quad
\text{for $T_n \leq t < T_{n+1}$.}
\end{equation}
\end{parag}

Estimates \eqref{e:decaysigma} and \eqref{e:decay2} above
correspond to \eqref{e:decay} in~\S\ref{ss:ssc}.
Therefore, arguing as in the final step of the proof
of Theorem~\ref{t:main}, we obtain the following result.

\begin{theorem}\label{t:ass}
Given $s>0$ and $1\leq p\leq \infty$, under Assumptions~\ref{ass2} 
and~\ref{ass3}, there exist a bounded, divergence-free velocity field $u$ and a
solution $\rho$ of the Cauchy problem for~\eqref{e:continuity} in $\R^2$, such
that $u$ is bounded in $\dot{W}^{s,p}(\R^2)$ uniformly in time,~$u$ and $\rho$
are supported in $\Qone$ for all times, and the functional mixing scale of
$\rho$ exhibits the following behavior:
%%%%%
\begin{itemize}
[leftmargin=30pt, itemsep=0pt]
\item
 case $s<1$: perfect mixing in finite time;
\item
case $s=1$: exponential decay;
\item
case $s>1$: polynomial decay.
\end{itemize}
\end{theorem}

\pagebreak[0]
\begin{remark}%\label{r:qssfinal}
\quad
%%%%%
\begin{enumerate}
[label=(\roman*), ref=(\roman*), 
leftmargin=0pt, itemsep=2pt, itemindent=30pt]

\item 
Thanks to Lemma~\ref{l:fg} we can deduce that the geometric mixing scale
 of $\rho(t,\cdot)$ exhibits the same behavior as the
functional mixing scale.

\item 
In view of Remark~\ref{r:tay}\ref{r:tay:4}, the fact that the
velocity field and the solution are supported in $\Qone$
implies the validity of Theorem~\ref{t:ass} on 
the torus $\T^2$.

\item 
For later use (in the companion paper~\cite{loss}), we make here
and in \ref{r:qssfinal.4} below some additional observations.
In the case $s=1$, every $\dot{H}^{-r}$ norm of $\rho$ decays
exponentially in time for $0 < r < 2$. 
Moreover,if $M_{\tilde r}$ (defined in~\eqref{e:Msigma}) is 
finite for some $\tilde r \geq 2$, then the $\dot{H}^{-r}$
norm of $\rho$ decays exponentially in time for~$0 < r < \tilde r$.

\item 
\label{r:qssfinal.4}
In Section~\ref{snake} we will verify Assumptions~\ref{ass2} 
and~\ref{ass3} for any~$s$ and~$p$ and construct a velocity field $u$ 
and a solution $\rho$ that are actually smooth in both time and space.
As a consequence of Theorem~\ref{t:ass}, this velocity field is bounded 
in~$\dot{W}^{1,p}(\R^2)$ uniformly in time and the functional mixing 
scale of $\rho$ decays exponentially.
Additionally,  this velocity field satisfies
\[
\| u(t,\cdot) \|_{\dot{W}^{r,p}(\R^2)}
\leq
C_r  \left( \lambda^{1-r} \right)^t 
\,,
\]
for any real number $r \geq 0$.
The estimate above follows from the proof of Lemma~\ref{l:fieldqss}
(recalling that in this case~$\tau=1$).
In particular, the Sobolev norms of $u$ of order higher than one grow
exponentially in time, while the Sobolev norms of order lower than one 
decay exponentially in time.

\item 
Finally, a reparametrization of the time variable in the example
constructed in Section~\ref{snake} gives a bounded, compactly supported,
divergence-free velocity field~$u$ such that~$u(t,\cdot) \in {\rm Lip}(\R^2)$
for almost every $0 \leq t \leq 1$, and such that the Cauchy problem for the
continuity equation associated to this velocity field admits non-unique 
solutions.
Indeed, its Lipschitz norm  blows up as $t \downarrow 0$ in such a way 
that the velocity field fails to belong to $L^1([0,1]; {\rm Lip} (\R^2))$. 
%just below the uniqueness class for the DiPerna-Lions theory
%in its $BV$ extension \cite{Amb}. 
This example improves on the result in~\cite{depauw} in the $BV$ case
(see also \cite{bressan}, \cite{llnmd}).
\end{enumerate}
\end{remark}

%
%	SECTION 7
%
\section{Second geometric construction}
\label{geo2}

In this section we describe another geometric construction of
divergence-free velocity fields $u$ together with
(non-trivial) solutions $\rho$ of the continuity
equation~\eqref{e:continuity}. The main improvement obtained by this 
approach is that we construct solutions that are smooth.
For paths and curves we follow the notation introduced
in Section~\ref{basic}. 

\medskip

We begin with a simple remark. Let $I$ be an open time interval
and $D$ an open subset of $\R^2$, and let
$\{\Phi(t,\cdot): \, t\in I\}$ be an area-preserving
flow on $D$ of class~$C^k$ with~$k\ge 2$. In other words,
$\Phi:I\times D\to\R^2$ is a map of class $C^k$
such that, for every $t\in I$,
$\Phi(t,\cdot)$ is diffeomorphism from $D$
onto an open set $\Omega(t)$,
which satisfies 
\[
J\Phi(t,z) := \det(\nabla\Phi(t,z)) =1.
\]

We denote by $\Omega$ the (open) set of all points
$(t,x)$ with $t\in I$ and $x\in\Omega(t)$.
 is then well known that the velocity field
$w:\Omega\to\R^2$ defined by
\begin{equation}
\label{e:velocity}
w(t,x) := \bd_t \Phi(t,z)
\quad\text{with }
x=\Phi(t,z)
\end{equation}
is of class $C^{k-1}$ and divergence-free.

Moreover, given a bounded function $\bar\rho$ on $D$,
the function $\rho: \Omega\to\R$ obtained by
transporting $\bar\rho$ with the flow $\Phi$, that is,
\[
\rho(t,x) := \bar\rho(z)
\quad\text{with }
x=\Phi(t,z)\,,
\]
is a weak solution of the transport equation \eqref{e:transport} 
(which agrees with the continuity equation \eqref{e:continuity}, 
since  the velocity is divergence-free).

\medskip

In the next proposition, we extend this result
in order to obtain a velocity field and a solution
defined on $I\times\R^2$, rather than on $\Omega$.

\begin{proposition}
\label{e:tronc}
Let $D$ be a simply-connected domain in $\R^2$, and let $\Phi$ 
be an area-preserving flow on $D$ of class~$C^k$,~$k\ge 2$. 
Let $D'$ be a closed subset of $D$.
Then there exists a divergence-free velocity field
$u:I\times\R^2\to\R^2$ of class $C^{k-1}$ such that
\begin{equation}
\label{e:compat8}
u(t,x)
= w(t,x)
= \bd_t \Phi(t,z),
\quad\text{if $x=\Phi(t,z)$ for some $z\in D'$.}
\end{equation}
Given  $\bar\rho:D'\to\R$ bounded,
the function $\rho:I\times\R^2\to\R$ defined by
\begin{equation}
\label{e:solution}
\rho(t,x) :=
\begin{cases}
  \bar\rho(z) & \text{if $x=\Phi(t,z)$ for some $z\in D'$,} \\
            0 & \text{otherwise,}
\end{cases}
\end{equation}
is a weak solution of the continuity equation
\eqref{e:continuity}.
\end{proposition}

\begin{remark}
The assumption that $D$ is simply connected
can be weakened, but not entirely removed.
Indeed, take $D:=\R^2\setminus\{0\}$ and let
$\{\Phi(t,\cdot): \, t\ge 0\}$ be the flow on $D$
associated with the (autonomous) velocity field $w(x):=x/|x|^2$.
Since $w$ is divergence-free on~$D$,
the flow is area preserving.
Consider now a curve $\Gamma$ 
that winds around the origin once counterclockwise. 
Then the flux through $\Gamma$ of any divergence-free
velocity field $u$ defined on $\R^2$ must be $0$, while
the flux of $w$ is $2\pi$, since the
distributional divergence of $w$ on $\R^2$ is $2\pi\, \delta_0$,
where~$\delta_0$ is  the Dirac mass at the origin.
This simple example shows that \eqref{e:compat8} cannot hold, 
if $\Phi(t,D')$ contains such a curve $\Gamma$ for some time $t$.
\end{remark}

Informally, $u$ is obtained by truncating $w$ on $D'$ and extending 
by zero. 
The difficulty in doing so is ensuring the divergence-free condition.
As customary to circumvent this problem, we truncate instead a potential
of $w$. We let $w$ be given by \eqref{e:velocity}, and choose a potential
$\phi$ for $w$.
We then multiply this potential by a suitable cut-off function, 
which agrees with $1$ on $D'$, and define $u$ as the velocity associated
to the new potential, which is automatically divergence-free.
We now present the proof in detail.

\begin{proof}[Proof of Proposition~\ref{e:tronc}]
We begin by selecting a smooth cut-off function
$g:\R^2\to [0,1]$ that agrees with $1$ on a open neighborhood of
$D'$ and has support contained in $D$.
We choose a point $z_0\in D$, which will be used to normalize the potential.
Since $D$ is simply connected, $\Omega(t)$ is simply
connected for every $t\in I$, and consequently the divergence-free
velocity field $w(t,\cdot)$ admits a unique potential
$\phi(t,\cdot)$ in the sense of \S\ref{s:potential} that satisfies  the
normalization condition
\begin{equation}
\label{e:renorm}
\phi(t,x_0(t))=0,
\quad\text{where}\quad
x_0(t):=\Phi(t,z_0)
\, .
\end{equation}
We then define the truncated potential
$\varphi(t,\cdot):\R^2\to \R$ by
\begin{equation}
\label{e:truncation-g}
\varphi(t,x) :=
\begin{cases}
  \phi(t,x) \, g(z)
    & \text{if $x=\Phi(t,z)$ for some $z\in D$, } \\
  0 & \text{otherwise,}
\end{cases}
\end{equation}
and finally take $u:= \nabla^\perp\varphi$.

Since $\Phi$ is of class $C^k$, both $w$  and $\phi$ are of class
$C^{k-1}$ in both variables, and $\phi$ is of class $C^k$ in $x$. 
Clearly the same holds for
$\varphi$, which in turn implies that $u$ is of class $C^{k-1}$.
Moreover, $\varphi$ agrees by construction with $\phi$
on an open neighborhood~$U$ of the set
of all points $\Phi(t,z)$ with $t\in I$, $z\in D'$,
and therefore $u$ agrees with $w$ on $U$.
In particular, \eqref{e:compat8} holds.

Next, we observe that  $\rho$ is obtained by transporting $\bar\rho$
with the flow $\Phi$, and hence it solves the continuity
equation $\bd_t\rho +\dive(w\rho)=0$ on $\Omega$.
On the other hand, $u$ and $v$ agree on $U$,
which contains the support of $\rho$, and therefore $\rho$
solves the continuity equation $\bd_t\rho +\dive(u\rho)=0$ in $\R^2$
as well.
\end{proof}

Let $\Gamma$ be a curve in the plane.
In the next lemma, we modify the definition of
the parametrization $\Psi$ of the tubular
neighborhood $B(\Gamma,r)$ given in \eqref{e:param}, in order to obtain
an area-preserving map.

\begin{lemma}
\label{s:modtub}
Let $\Gamma$ be a proper curve
parametrized by a path $\gamma:\R\to\R^2$ of class
$C^k$ with $k\ge 3$, such that
$| \dot\gamma(\cdot) | =\ell $ for some constant~$\ell$
and the tubular radius $\bar r$ of $\Gamma$ is strictly positive.
Let $r$ be a positive number such that $r \leq \ell \bar r / 2$ and
let $\Phi:\R\times(-r,r)\to\R^2$
be the map defined by
\begin{equation}
\label{e:param2}
\Phi(s,y):= \gamma(s) + \alpha(s,y / \ell) \, \eta(s)
\quad\text{with}\quad
\alpha(s,y'):= \frac{ 2 y' }{ 1 + \sqrt{1 -2   y'  \kappa(s)} }
\, .
\end{equation}
Then $\gamma(\cdot)=\Phi(\cdot,0)$ and
$\Phi$ is an area-preserving diffeomorphism of
class $C^{k-2}$, the image of which is contained in the
tubular neighborhood $B(\Gamma,2r/\ell)$ and contains~$B(\Gamma,r/(2\ell))$.
\end{lemma}

\begin{proof}
Using the assumption on $r$ and the fact that the
tubular radius $\bar r$ is no larger than the
curvature radius $1/|\kappa|$ of the curve, it follows 
that $r\le \ell/(2|\kappa|)$, which implies that
$\Phi$ is well defined on $\R\times(-r,r)$.

We observe that $\alpha$ is a function of class $C^{k-2}$
because $\kappa$ is of class $C^{k-2}$, and that
\begin{equation}
\label{e:param3}
\Phi(s,y) = \Psi\big(s,\alpha(s,y/\ell)  \big)
\quad\text{for every $s,y$,}
\end{equation}
where $\Psi$ is defined in \eqref{e:param}.
Since $\Psi$ is a diffeomorphism of class $C^{k-1}$
on $\R\times (-\bar r,\bar r)$ and the function
$y \mapsto \alpha(s,y/\ell)$ has strictly positive
derivative for every $s$ and maps~$(-r,r)$ into
$(-\bar r,\bar r)$,  $\Phi$ is a
diffeomorphism of class $C^{k-2}$.

The fact that $\Phi$ is area-preserving,
that is, $J\Phi=1$ everywhere, can be verified
by a direct computation.
For this purpose, it is convenient to write the gradient of
$\Phi$ at $(s,y)$ using the canonical basis
of $\R^2$ for the domain, and the orthonormal basis
$\tau(s),\eta(s)$, associated to the foliation of the tubular
neighborhood induced by $\Gamma$, for the codomain.
This choice gives that
\[
\nabla \Phi(s,y) =
\begin{pmatrix}
  \ell(1+ \kappa\, \alpha) & 0 \\
  \bd_s\alpha &  \frac{1}{\ell} \bd_{y'}\alpha
\end{pmatrix}
\, ,
\]
where $\kappa = \kappa(s)$ and $\alpha=\alpha(s,y/\ell)$.

Finally, the fact that the image of $\Phi$ is contained
in $B(\Gamma,2r/\ell)$ and contains $B(\Gamma,r/(2\ell))$
follows from Formula~\eqref{e:param3} and the estimate
$r/(2\ell) \le \alpha\le 2r/\ell$.
\end{proof}

In the next subsections, we associate a velocity field $u$
and a solution $\rho$ of the continuity equation \eqref{e:continuity}
to a given time-dependent proper curve $\Gamma$.
This construction will provide the building blocks for the example
described in the next section.

\begin{parag}[Velocity field associated to a time-dependent curve]
\label{s:u}

Let  $\Gamma$ be a time-dependent curve
parametrized by a path $\gamma:I\times\R\to\R^2$
of class~$C^k$ with $k\ge 3$. 
Let  $r$ be a positive number such that
%%%%%
\begin{enumerate}
[label=(\alph*), ref=(\alph*),
leftmargin=30pt, itemsep=2pt]

\item\label{s:u.1}
$|\dot\gamma(t,\cdot)|$ is equal to some $\ell(t)>0$
for every $t\in I$; 

\item\label{s:u.2}
$2r/\ell(t)$ is smaller or equal than the tubular radius 
of $\Gamma(t)$ for every $t\in I$.
\end{enumerate}
For every $t\in I$ we let $\Phi(t,\cdot)$ be the
diffeomorphism on $D:=\R\times(-r,r)$ defined by \eqref{e:param2},
and take the velocity field $u:I\times\R^2\to\R^2$
as in Proposition~\ref{e:tronc},
having set $D':=\R\times[-r/2,r/2]$.

A close inspection of the proof of Proposition~\ref{e:tronc}
shows that the construction of $u$ depends on the choice of the
point $z_0$ in $\R\times(-r,r)$, used in the normalization
condition \eqref{e:renorm}, and on the choice of the cut-off function $g$.
For the construction at hand, we choose: 
%%%%%
\begin{enumerate}
[label=(\alph*), ref=(\alph*), start=3,
leftmargin=30pt, itemsep=2pt]

\item\label{s:u.3}
$z_0:=(0,0)$;

\item\label{s:u.4}
$g(s,y):=\bar g(y/r)$,
where $\bar g:\R\to[0,1]$ is a fixed smooth function
that is \emph{even}, takes value $1$ in a neighborhood of $[-1/2,1/2]$,
and its support is contained in $(-1,1)$.
\end{enumerate}
\end{parag}

\begin{parag}[Canonical solution associated to a time-dependent curve]
\label{s:rho}
We fix an \emph{even} bounded function $\bar\rho:[-1/2,1/2]\to\R$
with zero integral over space and let~${\rho:I\times\R^2\to\R}$ be the solution of
the continuity equation \eqref{e:continuity} obtained by
replacing the function $\bar\rho(z)$ in formula \eqref{e:solution}
with $\bar\rho(y/r)$, that is,
\begin{equation}
\label{e:solution2}
\rho(t,x) :=
\begin{cases}
  \bar\rho(y/r) & \text{if $x=\Phi(t,s,y)$ for some $s\in\R$, $y\in[-r/2,r/2]$,} \\
            0 & \text{otherwise.}
\end{cases}
\end{equation}
\end{parag}

\begin{remark}\label{r:boh}
\quad
%%%%%
\begin{enumerate}
[label=(\roman*), ref=(\roman*), 
leftmargin=0pt, itemsep=2pt, itemindent=30pt]

\item
\label{r:boh.1} 
The velocity field $u$ constructed above is uniquely determined 
by the choice of the parametrization $\gamma$, the number $r$, 
and the function $\bar g$.
Since~$\bar g$ is fixed for the rest of the paper, the relevant 
parameters are therefore $\gamma$ and $r$.

\item
\label{r:boh.2} 
The solution $\rho$ depends only on purely geometric quantities,
and not on the choice of the parametrization $\gamma$.
More precisely, using formulas~\eqref{e:param2} and~\eqref{e:solution2}
one readily checks that the value $\rho(t,x)$ is zero if
$\mathrm{dist} (x,\Gamma(t)) > r/\ell(t)$, and otherwise it depends on:
%%%%%
\begin{itemize}
\item $r$, $t$, and $\ell(t)$;
\item the distance $\mathrm{dist} (x,\Gamma(t))$;
\item the curvature of $\Gamma(t)$ at the projection of $x$ on $\Gamma(t)$.
\end{itemize}

\item 
\label{r:boh.3} 
By construction, for every $t\in I$, the velocity field $u(t,\cdot)$ is 
supported in~$\Phi(t,\R\times(-r,r))$, which is contained in
$B(\Gamma(t),2r/\ell(t))$, while $\rho(t,\cdot)$ is supported in
$\Phi(t,\R\times(-r/2,r/2))$, which is contained in $B(\Gamma(t),r/\ell(t))$
(cf.~Lemma~\ref{s:modtub}).

\item 
\label{r:boh.4}
It follows from Formula~\eqref{e:solution2} that the  solution~$\rho(t,\cdot)$
has zero integral over space, since the initial data $\bar \rho$ is assumed
to have the same property and the change of variable $\Phi$ is area preserving.
This property will be used in Section~\ref{snake}.
\end{enumerate}
\end{remark}

We suppose now that we are given two time-dependent curves $\Gamma$ and
$\longtilde{\Gamma}$, and we let $u, \tilde u$ and $\rho, \tilde\rho$
be, respectively,  the corresponding velocity fields and solutions
constructed in~\S\ref{s:u} and \S\ref{s:rho}.
In the next section, we will exploit a kind of localization principle,
stating that,
if $\Gamma$ and $\longtilde{\Gamma}$ agree in a
neighborhood of some point~$x_0$,
then $u$, $\tilde u$ and
$\rho$, $\tilde\rho$ also agree in a neighborhood of $x_0$.

In Lemma~\ref{s:locality} we give a precise statement of this principle,
specifically designed for the applications described in the next section.

We first introduce some additional notation.

\begin{parag}[Centered sub-arcs and curved rectangles]
Let $\Gamma$ be a curve parametrized by a path
$\gamma$ such that $|\dot\gamma(\cdot)| = \ell$ constant,
 and let  $x_0=\gamma(s_0)$ be a point on~$\Gamma$. For a given $\delta>0$,
we denote by $I(\Gamma,x_0,\delta)$ the
(centered) sub-arc given by all~$x\in\Gamma$ such that their geodesic
distance from $x_0$ is strictly less than $\delta$.  That is,
\[
I(\Gamma,x_0,\delta)
:= \gamma \big( (s_0 - \delta/\ell,s_0 + \delta/\ell) \big)
\, .
\]
Moreover, given a $\delta'>0$ that  is no larger  than the tubular
radius of $\Gamma$, we denote by $R(\Gamma,x_0,\delta,\delta')$
the (open, centered) curved rectangle given by  all $x\in\R^2$ such that
their distance from $\Gamma$ is strictly less than $\delta'$ and their
projection on $\Gamma$ belongs to $I(\Gamma,x_0,\delta)$.
That is,
\[
R(\Gamma,x_0,\delta,\delta')
:= \Psi \big( (s_0-\delta/\ell,s_0+\delta/\ell)\times(-\delta',\delta') \big)
\, ,
\]
where again $\Psi$ is defined in \eqref{e:param}.
\end{parag}

\begin{lemma}
\label{s:locality}
Let $\Gamma$ and $\longtilde\Gamma$ be two time-dependent, proper curves
of class~$C^k$ with $k\ge 3$, parametrized by $\gamma, \tilde\gamma:I\times\R\to\R^2$, 
respectively.
Assume that~\ref{s:u.1} and~\ref{s:u.2} in~\S\ref{s:u}
are verified by $\gamma$ and $\tilde\gamma$
with the same $\ell:I\to(0,+\infty)$ and the same $r>0$. Let
$u, \tilde u$ be defined as in~\S\ref{s:u} and $\rho, \tilde\rho$
be defined as in~\S\ref{s:rho}.
Assume in addition that there exist $\delta>0$ and $s_0 \in \R$ such that,
for every~$t\in I$,
%%%%%
\begin{enumerate}
[label={\rm(\alph*)}, ref=(\alph*),
leftmargin=30pt, itemsep=2pt]

\item%\label{s:locality.1}
$\gamma(t,s_0) = \tilde\gamma(t,0) = : x_0(t)$;

\item\label{s:locality.2}
the sub-arcs $I \big( \Gamma(t),x_0(t),\delta \big)$
and $I \big( \longtilde\Gamma(t),x_0(t),\delta \big)$
coincide and have the same orientation;

\item\label{s:locality.3}
denoting by $v_\nn$  the normal velocity of $\Gamma$, we have
\[
\int\limits_{\gamma(t,[0,s_0])}
\hskip -5 pt v_\nn(t,x) \, d\sigma(x) = 0
\,.
\]
\end{enumerate}

Then $u(t,x) = \tilde u(t,x)$
and $\rho(t,x) = \tilde\rho(t,x)$
for every $t\in I$ and every $x$ in
the curved rectangle
$R(t) := R \big( \Gamma(t),x_0(t),\delta,2r/\ell(t) \big)$.
\end{lemma}

The proof is not difficult, but we must revisit
the entire construction of~$u$ and~$\rho$,
which is divided between \S\ref{s:u}, \S\ref{s:rho},
and the proofs of Proposition~\ref{e:tronc}
and Lemma~\ref{s:modtub}.

\begin{proof} 
We fix $t\in I$.
Using that \ref{s:u.1} and \ref{s:u.2} are assumed verified, and 
the fact that $\gamma(t,\cdot)$ and~$\tilde\gamma(t,\cdot)$
have the same parametrization speed $\ell(t)$, we obtain that
\[
\gamma(t,s+s_0) = \tilde\gamma(t,s),
\quad\text{when $|s|\le\delta/\ell$.}
\]
From  this identity, it readily follows
that the flows $\Phi$ and $\longtilde{\Phi}$
defined by \eqref{e:param2} satisfy
\begin{equation}
\label{e:7.7.0g}
\Phi(t,s+s_0,y) = \longtilde{\Phi}(t,s,y),
\quad\text{when $|s|\le\delta/\ell$, $|y|<r$,}
\end{equation}
which implies that the velocity fields $w$
and $\longtilde w$ defined by \eqref{e:velocity} satisfy
\begin{equation}
\label{e:7.7g}
w(t,x) = \longtilde w(t,x),
\quad\text{when $x\in U(t)$,}
\end{equation}
where
\[
U(t)
:= \big\{ \longtilde \Phi(t,s,y) \, : \ |s|\le\delta/\ell \, , \ |y|<r \big\}
\, .
\]

\smallskip
We let now $\phi(t,\cdot)$ and $\longtilde\phi(t,\cdot)$ be the potentials
of $w(t,\cdot)$ and $\longtilde w(t,\cdot)$, respectively,
constructed in the proof of Proposition~\ref{e:tronc}.
We claim that
\begin{equation}
\label{e:7.8g}
\phi(t,x) = \longtilde\phi(t,x)
\quad\text{when $x\in U(t)$.}
\end{equation}
Since the corresponding fields agree on $U$ and $U$ is connected, 
it suffices to show that these potentials agree at one point in $U$. 
We will show that they both vanish at $x_0(t)$.
Indeed, formula \eqref{e:renorm}, Assumption \ref{s:u.3} in
\S\ref{s:u}, and the identities $\Phi(t,0,0) = \gamma(t,0)$
and $\longtilde\Phi(t,0,0) = \tilde\gamma(t,0)$
yield
\[
\phi \big( t,\gamma(t,0) \big)
= \longtilde\phi \big( t, \tilde\gamma(t,0) \big)
= 0
\,.
\]
Since $\tilde\gamma(t,0) = x_0(t)$,
it follows that
\[
\tilde\phi(t,x_0(t)) = 0
\, .
\]
Finally, since $x_0(t) = \gamma(t,s_0)$,
taking into account again~\ref{s:u.3} and the identity
$v_\nn=w\cdot\eta$, we obtain 
\begin{align*}
\phi(t,x_0(t))
= \phi(t,\gamma(t,s_0))
& = \phi(t,\gamma(t,s_0)) - \phi(t,\gamma(t,0)) \\
& = \hspace{-8pt} \int\limits_{\gamma(t,[0,s_0])}
    \hspace{-8pt} w\cdot\eta \, d\sigma
  = \hspace{-8pt} \int\limits_{\gamma(t,[0,s_0])}
   \hspace{-8pt} v_\nn \, d\sigma
  = 0 \, .
\end{align*}
The proof of \eqref{e:7.8g} is complete.

The rest of the proof is straightforward.
From \eqref{e:7.8g} and the choice of the cut-off function $g$
(see \ref{s:u.4} in \S\ref{s:u}), we have that the
truncated potentials~$\varphi(t,\cdot)$ and $\tilde\varphi(t,\cdot)$,
defined by \eqref{e:truncation-g}, agree on $U(t)$.
Furthermore, one can show that both potentials vanish
on $R(t)\setminus U(t)$, and therefore they agree on
the whole $R(t)$, which implies that the corresponding
velocity fields $u(t,\cdot)$ and $\tilde u(t,\cdot)$
agree on~$R(t)$.

\smallskip
It remains to show that
\begin{equation}
\label{e:7.11g}
\rho(t,x)=\tilde\rho(t,x)
\quad\text{when $x\in R(t)$,}
\end{equation}
but this fact follows from Remark~\ref{r:boh}\ref{r:boh.2}.
\end{proof}

%
%	SECTION 8
%

\section{Second example: Peano snake}
\label{snake}

In this final section, we verify Assumptions~\ref{ass2} and \ref{ass3}
for any $s$ and $p$ and construct a specific example of
quasi-self-similar evolution. 
By doing so, we validate the assumptions of Theorem~\ref{t:ass}. 
In particular, we establish the existence of a bounded, divergence-free 
velocity field supported in the unit square, which is Lipschitz
continuous uniformly in time, and the existence of a solution of the
continuity equation~\eqref{e:continuity} with the property that its 
functional and geometric mixing scales decay exponentially in time.
We call this example  the ``Peano snake", since the construction is 
reminiscent of the iterative construction of the Peano curve 
(cf.~Figure~\ref{f:ideatwo}).

The velocity field and the solution that we construct
are smooth in both time and space (see Remark~\ref{r:smoothC}).
However, any Sobolev norm of order higher than one is not bounded 
uniformly in time.

We proceed as follows. 
In~\S\ref{ss:keycomb}, we first describe the combinatorial
structure of our example. 
Using the tools provided in Section~\ref{geo2}, we  then prove 
in~\S\ref{ss:contruC}, \S\ref{s:solse}, and~\S\ref{s:regfin} that the 
construction of a basic family verifying Assumptions~\ref{ass2} and \ref{ass3}
can be reduced to the construction of two time-dependent, proper curves 
satisfying a certain number of geometric conditions. 
These conditions are given in~\S\ref{p:basiccurvesC} 
and~\S\ref{r:equivalenceC}.
Finally in~\S\ref{s:costruzione-g}, \S\ref{s:costruzione-g1}, 
and \S\ref{s:costruzione-g2}, we present the actual construction 
of the two curves.

\begin{parag}[Combinatorial structure]
\label{ss:keycomb}
We begin by illustrating the combinatorial structure
of this quasi-self-similar example. The complete construction
is rather complex and the purpose of this subsection
is to provide a graphical representation of the solution for 
the first two steps in the iteration, in order to help the reader 
visualize our construction. 
We omit all details that are not needed for this purpose.

The starting point is a basic family (in the sense of Assumption~\ref{ass2}) 
consisting of six pairs of velocity fields~$u_j$ and solutions $\rho_j$.
At this stage, we do not directly define~$u_j$ and~$\rho_j$,
rather we describe the supports of the solutions $\rho_j$,
which we denote by~$E_j$. In fact, we describe the sets $E_j$
only for~$j=1,2$  and only at the 
initial time $t=0$ and at the final time $t=1$ 
(Figure~\ref{f:idea}). 
The sets $E_j$ for  $j=3,\dots,6$ 
are obtained by means of appropriate rotations.

\medskip
\begin{figure}[ht]
\begin{center}
  \includegraphics{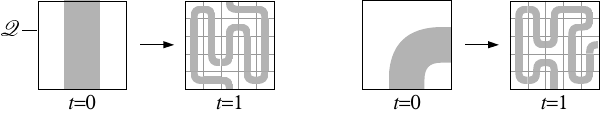}
  \caption{An example of basic family for a quasi-self-similar
  construction: the sets in gray are the supports $E_j$ 
  of the solutions~$\rho_j$ for $j=1$ (left) and $j=2$ (right)
  and at times $t=0$ and $t=1$.}
  \label{f:idea}
\end{center}
\end{figure}

As depicted in Figure~\ref{f:idea}, at time $t=0$ the sets $E_1$
and $E_2$ are, respectively, a straight strip (Figure~\ref{f:idea}, left) 
and a bent strip  (Figure~\ref{f:idea}, right), 
while at time $t=1$ they are composed by $25$ rotated and translated 
copies of the two basic elements, scaled by a factor $\lambda=1/5$.

\medskip

To obtain the velocity $u$ and solution $\rho$, we implement the 
construction in~\S\ref{ss:ass} with $\bar \lambda = 1/2$. 
In the first step of the construction, we must choose a pair $(u_j,\rho_j)$ 
from the basic family for every square in the tiling $\Til_{1/2}$. 
Our choice is such that at time $t=0$ 
the corresponding sets $E_j$ are all bent strips, and 
are patched together to create the (almost round) annulus shown 
in Figure~\ref{f:ideatwo}, left.

This annulus is the support of the initial data $\bar{\rho}$.
Using Figure~\ref{f:idea} we can draw the support of the solution $\rho$ 
at time $t=T_1=1$ (Figure~\ref{f:ideatwo}, middle), then at 
time $t=T_2=1+\tau$ (Figure~\ref{f:ideatwo}, right), and so on.

\begin{figure}[ht]
\begin{center}
  \includegraphics[scale=1]{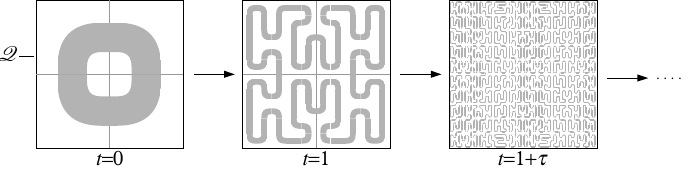}
  \caption{Initial choice of the four basic elements and further steps 
  in the quasi-self-similar evolution: the grayed set is the support
  of the solution at each step.}
  \label{f:ideatwo}
\end{center}
\end{figure}

While Figures~\ref{f:idea} and \ref{f:ideatwo} give an accurate illustration
of the combinatorial structure of the example that will be presented in 
full detail in the rest of this section, a rigorous justification of the
example requires a careful analysis, as  we need to produce a smooth solution
$\rho$ that is transported by a smooth, divergence-free  velocity field $u$.
\end{parag}

\begin{parag}[Conditions on the basic curves]
\label{p:basiccurvesC}
The fundamental ingredient of our construction will be
two time-dependent proper curves $\Gamma_1(t)$ and $\Gamma_2(t)$, 
corresponding to the two sets $E_1$ and $E_2$ described in~\S\ref{ss:keycomb},
with parametrizations~$\gamma_1$,~$\gamma_2 : [0,1] \times \R \to \R^2$ of
class $C^\infty$, such that, for every $t\in [0,1]$:
%%%%%
\begin{enumerate}
[label=(\alph*), ref=(\alph*),
leftmargin=30pt, itemsep=2pt]

\item\label{p:basiccurvesC.1}
$\gamma_1(t,0) = \gamma_2(t,0) = (0,-1/2)$, $\gamma_1(t,1) = (0,1/2)$,
and $\gamma_2(t,1) = (1/2,0)$;

\item\label{p:basiccurvesC.2}
there exists a constant $\delta>0$ such that outside the square
$(1-\delta)\Qone$, homothetic to $\Qone$, each of the curves 
$\Gamma_1(t)$ and $\Gamma_2(t)$ agrees with two unbounded
half-lines, orthogonal to $\bd\Qone$ and passing through 
the points defined in (a);

\item\label{p:basiccurvesC.3}
$| \dot \gamma_1 (t,s)| = | \dot \gamma_2 (t,s)| =: \ell(t)$
for every $s\in\R$, 
and in particular the intersections of the curves $\Gamma_1(t)$
and $\Gamma_2(t)$ with $\Qone$ have length $\ell(t)$;

\item\label{p:basiccurvesC.4}
denoting by $v_\nn^1$ and $v_\nn^2$, the normal velocity of 
$\Gamma_1$ and $\Gamma_2$, respectively, then 
\[
  \int\limits_{\gamma_1(t,[0,1])} \hspace{-8pt} v_\nn^1(t,x) \, d\sigma(x) 
= \hspace{-8pt}
  \int\limits_{\gamma_2(t,[0,1])} \hspace{-8pt} v_\nn^2(t,x) \, d\sigma(x) 
= 0 
\,;
\]

\item\label{p:basiccurvesC.5}
for every square $Q$ in the tiling $\Til_{1/5}$ of $\Qone$, 
the sub-arc $\Gamma_1(1) \cap Q$ can be written as a translated,
rescaled, and possibly rotated copy of $\Gamma_1(0) \cap \Qone$ 
or~$\Gamma_2(0) \cap \Qone$; the same holds for $\Gamma_2(1) \cap Q$.
\end{enumerate}
\end{parag}

\begin{parag}[Simplified geometric conditions]
\label{r:equivalenceC}
In this paragraph we replace some of the conditions in~\S\ref{p:basiccurvesC} 
by purely geometric ones, that is, conditions that
are written in terms of the curves $\Gamma_1$ and $\Gamma_2$ 
and do not involve the parametrizations~$\gamma_1$
and~$\gamma_2$. In \S\ref{s:costruzione-g}, \S\ref{s:costruzione-g1}, 
and~\S\ref{s:costruzione-g2}, we will then be able  to give the curves 
$\Gamma_1$ and $\Gamma_2$ without describing explicitly
the parametrizations $\gamma_1$
and $\gamma_2$.

More precisely, we consider the following alternative conditions,
in which we denote by $\ell_1(t)$ and $\ell_2(t)$
the length of the intersection of the curve $\Gamma_1(t)$ 
and~$\Gamma_2(t)$ with $\Qone$, respectively: 
%%%%%
\begin{enumerate}
[label=(\alph*'), ref=(\alph*'), start=3,
leftmargin=30pt, itemsep=2pt]

\item\label{p:basiccurvesC.3'}
$\ell_1(0) = \ell_2(0)$;
\end{enumerate}
%%%%%
\begin{enumerate}
[label=(\alph*''), ref=(\alph*''), start=3,
leftmargin=30pt, itemsep=2pt]

\item\label{p:basiccurvesC.3''}
the derivatives in $t$ of the functions $\ell_1(t)$
and $\ell_2(t)$ are strictly positive;
\end{enumerate}
%%%%%
\begin{enumerate}
[label=(\alph*'), ref=(\alph*'), start=4,
leftmargin=30pt, itemsep=2pt]

\item\label{p:basiccurvesC.4'}
the area of both connected components of
$\Qone \setminus \Gamma_1(t)$ equals $1/2$, and the same
holds for the area of the two connected components of
$\Qone \setminus \Gamma_2(t)$.
\end{enumerate}
%%%%%
%
We claim that Conditions \ref{p:basiccurvesC.2}, \ref{p:basiccurvesC.3'},
\ref{p:basiccurvesC.3''}, \ref{p:basiccurvesC.4'}, and \ref{p:basiccurvesC.5} 
imply Conditions \ref{p:basiccurvesC.1}-\ref{p:basiccurvesC.5} 
in \S\ref{p:basiccurvesC}.
First of all, \ref{p:basiccurvesC.1} follows from \ref{p:basiccurvesC.2}
by choosing suitable parametrizations~$\gamma_1(t,\cdot)$ and 
$\gamma_2(t,\cdot)$.
Next, we modify such parametrizations in such a way that 
$|\dot \gamma_1(t,\cdot)|$ and $|\dot \gamma_2(t,\cdot)|$ are constant 
for all~$t$. 
This fact, together with~\ref{p:basiccurvesC.1}, entails that 
$|\dot \gamma_1(t,s)| = \ell_1(t)$ and $|\dot \gamma_2(t,s)| = \ell_2(t)$.
Condition~\ref{p:basiccurvesC.3'}, together with condition~\ref{p:basiccurvesC.5},
implies that $\ell_1(1)=\ell_2(1)=5\ell_1(0)$.
Then condition~\ref{p:basiccurvesC.3''} implies that with a change
of variable in $t$ we can achieve $\ell_1(t) = \ell_2(t) =: \ell(t)$
for all $t$, that is, condition~\ref{p:basiccurvesC.3} holds. 
Finally \ref{p:basiccurvesC.4} and \ref{p:basiccurvesC.4'} are equivalent 
by Remark~\ref{s:rembasic}\ref{s:rembasic.3}. 
\end{parag}

\begin{parag}[A preliminary example]
\label{s:combin}
Two curves satisfying some, but unfortunately not all, of the conditions 
in~\S\ref{p:basiccurvesC} are depicted in Figure~\ref{f:comb1} (to be 
compared with Figure~\ref{f:idea}).
In this figure we can see the evolution of $\Gamma_1$ starting from a 
straight segment $\Gamma_1(0)$ inside~$\Qone$, and the evolution of
$\Gamma_2$ starting from a (slightly modified) quarter of circle
$\Gamma_2(0)$  inside~$\Qone$.
Note that both $\Gamma_1(1)$ and $\Gamma_2(1)$ can be written as unions
of $25$ copies of the segment $\Gamma_1(0)$ or of the (modified) quarter
of circle $\Gamma_2(0)$ which are scaled by a factor $1/5$ and suitably
rotated and translated.

\begin{figure}[ht]
\begin{center}
  \includegraphics[scale=1]{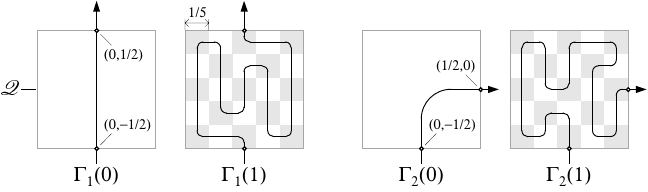}
  \caption{Example of curves satisfying  \ref{p:basiccurvesC.1}, \ref{p:basiccurvesC.2},
  and \ref{p:basiccurvesC.5} in~\S\ref{p:basiccurvesC}.}
  \label{f:comb1}
\end{center}
\end{figure}

From the geometry of $\Gamma_1$ and $\Gamma_2$ as shown in Figure
\ref{f:comb1}, it follows that these curves satisfy
Conditions~\ref{p:basiccurvesC.1}, \ref{p:basiccurvesC.2},
and \ref{p:basiccurvesC.5} (which is the most relevant).
However, Condition~\ref{p:basiccurvesC.3}
fails at $t=0$ and $t=1$, while
Condition~\ref{p:basiccurvesC.4} depends on the parametrization and
cannot be directly verified from the picture.
Moreover, these curves do not satisfy some of the alternative
conditions either, namely Conditions~\ref{p:basiccurvesC.3'}
and \ref{p:basiccurvesC.4'}, while Condition~\ref{p:basiccurvesC.3''}
could be in principle satisfied, at least for some choice of
the curves for $0<t<1$. 

In~\S\ref{s:costruzione-g}, \S\ref{s:costruzione-g1}, and 
\S\ref{s:costruzione-g2} we construct an example of basic curves 
satisfying the conditions in~\S\ref{p:basiccurvesC}
and~\S\ref{r:equivalenceC}. 
That construction is a modification of the present example 
(cf.~Figure~\ref{f:gamma-i+f}) and is significantly more complex.
\end{parag}

In the next subsections, we assume the existence of two basic curves
$\Gamma_1$ and~$\Gamma_2$ as in~\S\ref{p:basiccurvesC}
and~\S\ref{r:equivalenceC}. 
We then use the results in Section~\ref{geo2} to construct some 
associated velocity fields $u_1$ and $u_2$, and solutions $\rho_1$ 
and $\rho_2$ that satisfy Assumption~\ref{ass2}. 
Next, we obtain the velocity field $u$ and the solution $\rho$ by 
implementing the iterative procedure in~\S\ref{ss:ass}, according
to the combinatorial structure in~\S\ref{ss:keycomb}.
Afterwards, we  verify that $u$ satisfies Assumption~\ref{ass3}.
Lastly, we establish additional regularity properties
of $u$ and $\rho$.

\begin{parag}[Construction of $\boldsymbol{u_i}$ and $\boldsymbol{\rho_i}$]
\label{ss:contruC}
We choose a small $r>0$ such that:
%%%%%
\begin{itemize}
[leftmargin=30pt, itemsep=2pt]

\item 
$r < \min\{ \frac{1}{2} r_1(t) \ell(t) ,\, \frac{1}{2} r_2(t) \ell(t) \}$ 
for every $t \in [0,1]$,
where $r_1(t)$ and~$r_2(t)$ are the tubular radii of $\Gamma_1(t)$ and~$\Gamma_2(t)$,
respectively, and $\ell(t)$ is as in \S\ref{p:basiccurvesC}\ref{p:basiccurvesC.3};

\item 
$B(\Gamma_1(t),2r) \subset \Qone$ and $B(\Gamma_2(t),2r) \subset \Qone$
for every $t \in [0,1]$.
\end{itemize}
Such an $r$ exists due to the smoothness of $\Gamma_1$ and
$\Gamma_2$, and Condition~\ref{p:basiccurvesC.2} in \S\ref{p:basiccurvesC}.

Then we follow the steps described  in \S\ref{s:u} and \S\ref{s:rho} to
obtain the associated  time-dependent divergence-free velocity fields
$u_1$ and $u_2$ of class $C^\infty$, and the corresponding smooth 
solutions $\rho_1$ and $\rho_2$. 
One readily checks that:
%%%%%
\begin{enumerate}
[label=(\alph*), ref=(\alph*),
leftmargin=30pt, itemsep=2pt]

\item\label{ss:contruC.1}
$u_1(t,\cdot)$ and $\rho_1(t,\cdot)$ are supported in 
$B(\Gamma_1(t),2r/\ell(t))$, while $u_2(t,\cdot)$ and $\rho_2(t,\cdot)$ 
are supported in $B(\Gamma_2(t),2r/\ell(t))$;

\item\label{ss:contruC.2}
in particular, the supports of $u_1(t,\cdot)$ and $\rho_1(t,\cdot)$ intersect
$\bd \Qone$ inside the segments 
$(-\frac{2r}{\ell(t)} , \frac{2r}{\ell(t)} ) \times \{ -\frac{1}{2} \}$
and $( -\frac{2r}{\ell(t)} , \frac{2r}{\ell(t)} ) \times \{ \frac{1}{2} \}$,
while the supports of~$u_2(t,\cdot)$ and $\rho_2(t,\cdot)$ intersect
$\bd \Qone$ inside the segments 
$( -\frac{2r}{\ell(t)} , \frac{2r}{\ell(t)} ) \times \{ -\frac{1}{2} \}$ 
and $\{ \frac{1}{2} \} \times (-\frac{2r}{\ell(t)},\frac{2r}{\ell(t)} )$;

\item\label{ss:contruC.3}
$u_1(t,\cdot)$ and $u_2(t,\cdot)$ are tangent to the boundary $\bd\Qone$.
\end{enumerate}
As a matter of fact, Properties~\ref{ss:contruC.1} and \ref{ss:contruC.2} 
above follow at once from Remark~\ref{r:boh}\ref{r:boh.3}.

Next we show that $u_1$ is tangent to the segment 
$( -\frac{2r}{\ell(t)} , \frac{2r}{\ell(t)} ) \times \{ \frac{1}{2} \}$.

(The rest of Property~\ref{ss:contruC.3} above  can be proved in a similar way.)

We consider the vertical line $\tilde\Gamma(t)$ with parametrization
$\tilde\gamma(t,s) = (0,\frac{1}{2}+ s \ell(t))$ and let $\tilde u$ be the velocity
field associated to~$\tilde \Gamma$ as in \S\ref{s:u}. 
We now apply Lemma~\ref{s:locality} to~$\tilde\Gamma$ and $\Gamma_1$
with $s_0=1$ and $x_0(t)=(0,\frac{1}{2})$,%
\footnoteb{We observe that Assumption~\ref{s:locality.2} of that
lemma is satisfied with the same $\delta>0$  as in Condition~\ref{p:basiccurvesC.2}
in~\S\ref{p:basiccurvesC}, while assumption~\ref{s:locality.3} in that lemma
follows from Condition~\ref{p:basiccurvesC.4} in \S\ref{p:basiccurvesC};
all other assumptions in the lemma
are clearly satisfied by our choice of~$\Tilde\Gamma$.}
and we obtain that, for any $t\in [0,1]$, the velocity field $\tilde u(t,\cdot)$
coincides with $u_1(t,\cdot)$ in the  rectangle
\[
\textstyle R \big( \Gamma_1(t),(0,\frac{1}{2}),\delta,\frac{2r}{\ell(t)} \big)
=
\big( -\frac{2r}{\ell(t)} , \frac{2r}{\ell(t)} \big)
\times \big( \frac{1}{2} - \delta, \frac{1}{2}+\delta \big) 
\, .
\]
It is, therefore, sufficient to show that $\tilde u(t,\cdot)$ is tangent to
$\big( -\frac{2r}{\ell(t)} , \frac{2r}{\ell(t)} \big) \times \{\frac{1}{2} \}$
for any~$t\in [0,1]$. This property follows
from the fact that the diffeomorphism $\tilde \Phi$ in~\eqref{e:param2},
associated to $\tilde \Gamma$, has the form
\[
\tilde \Phi(t,y) 
= \left( \frac{y}{\ell(t)} \, , \, \frac{1}{2} + s \ell(t) \right)
\, ,
\]
and from the construction in \S\ref{s:u}.
\end{parag}

\begin{parag}[Verification of Assumption~\ref{ass2} 
for $\boldsymbol{u_i}$ and $\boldsymbol{\rho_i}$]
\label{s:solse}
First, we recall that the velocity fields $u_1$ and $u_2$
are smooth and tangent to the boundary $\bd\Qone$, and in particular satisfy
Assumption~\ref{ass2}\ref{ass2.1}, while Assumption~\ref{ass2}\ref{ass2.2}
follows from Remark~\ref{r:boh}\ref{r:boh.4}.
Assumption~\ref{ass2}\ref{ass2.3} can be obtained by combining the fact
that the solutions $\rho_1$ and $\rho_2$ can be described in
purely geometric terms (see Remark~\ref{r:boh}\ref{r:boh.2})
and the quasi-self-similarity of the curves~$\Gamma_1$ and~$\Gamma_2$
in Condition~\ref{p:basiccurvesC.5} in \S\ref{p:basiccurvesC}.
\end{parag}

\begin{parag}[Verification of Assumption~\ref{ass3} for $\boldsymbol{u}$]
\label{s:regfin}
We first observe that,  according to the general scheme for a quasi-self
similar construction (see \S\ref{ss:ass}) and to the
specific combinatorial structure of our example (see \S\ref{ss:keycomb}),
for $T_n < t \leq T_{n+1}$ and on each of the tiles in $\Til_{1/(2 \cdot 5^n)}$,
the velocity field $u$ agrees with one of the velocity fields $u_1$ and $u_2$
constructed in~\S\ref{ss:contruC} after a rescaling, a translation and 
possible rotation.

We next check that $u$ is smooth on the entire plane $\R^2$.
This result will imply Assumption~\ref{ass3} for any $s$ and $p$.
The regularity in the interior of each square of the tiling
$\Til_{1/(2 \cdot 5^n)}$ follows from the regularity of $u_1$ and $u_2$.
We are, therefore, left with checking the regularity across the boundary
of $\Qone$ and at the interface between each
pair of squares in the tiling.

The combinatorial structure in \S\ref{ss:keycomb} guarantees that, 
at every time, the rotated and rescaled translations of the curves 
$\Gamma_1$ and $\Gamma_2$ never intersect the boundary of~$\Qone$. 
Indeed, they always intersect the boundary of a tile that is 
adjacent to another tile, but  never at the boundary of $\Qone$. 
Together with Property~\ref{ss:contruC.2} in \S\ref{ss:contruC}, 
this observation implies that the velocity field $u$ vanishes 
in a neighborhood of the boundary of $\Qone$, and therefore 
is regular there. 

It remains to prove the regularity of $u$ at the interface $\Sigma$ 
between two adjacent squares $Q'$ and $Q''$ of $\Til_{1/(2 \cdot 5^n)}$.
We call $\Gamma'$ and $\Gamma''$ the rotated and rescaled translations 
of the curves $\Gamma_1$ and/or $\Gamma_2$ which lie inside $Q'$ and 
$Q''$, respectively (see Figure~\ref{f:gluing}). 
We have two possibilities: either the curves $\Gamma'$ and $\Gamma''$ 
do not intersect the interface $\Sigma$, or at least one of them does. 
In the first case $u$ is zero in a neighborhood of $\Sigma$ (again 
from Property~\ref{ss:contruC.2} in \S\ref{ss:contruC}), 
and therefore there is nothing to prove. 
In the second case, the combinatorial structure in \S\ref{ss:keycomb}
guarantees that both $\Gamma'$ and~$\Gamma''$ intersect $\Sigma$ at 
its mid point, which we denote by $\tilde x$.

\begin{figure}[ht]
\begin{center}
  \includegraphics[scale=1]{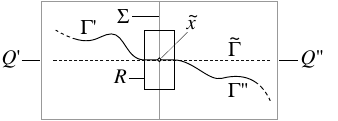}
  \caption{~}
  \label{f:gluing}
\end{center}
\end{figure}

Without loss of generality, we can always assume that the segment 
$\Sigma$ is vertical, as in Figure~\ref{f:gluing}.
Then we denote by $\widetilde\Gamma$ the horizontal straight line
passing through the mid point $\tilde x$ of $\Sigma$, parametrized by
$\tilde\gamma(s):=\tilde x+(s,0)$, and  by $\tilde u$ the
(autonomous) velocity field associated to the parametrization
$\tilde\gamma$ and the positive number $r/5^n$ as in \S\ref{s:u},
where~$r$ is chosen in \S\ref{ss:contruC}.
Moreover, we denote by $R$ the open rectangle centered at $\tilde x$
with width~$2\delta / 5^n$ and height $4r$
(see Figure~\ref{f:gluing} again).

We note that, by construction, the velocity field $u$ agrees in $Q'$
(up to a rotation, a translations and rescaling) with one of the basic 
velocity fields $u_1$ and $u_2$, and therefore it agrees with the 
velocity field associated to the parametrization of~$\Gamma'$ and 
the number $r/5^n$ as in \S\ref{s:u}.
Then Condition~\ref{p:basiccurvesC.4} in \S\ref{p:basiccurvesC} shows 
that we can apply Lemma~\ref{s:locality} and obtain that $u$ and 
$\tilde u$ agree in  $R\cap \Bar Q'$.
Since the same argument applies to the square $Q''$, we have that 
$u$ and $\tilde u$ agree on the whole rectangle $R$. 
Hence, $u$ is smooth as the restriction of $\tilde u$ to $R$.

We conclude the proof of the smoothness of $u$ (across the interface 
$\Sigma$) by noticing that, again by construction, $u$ vanishes in a 
neighborhood of the complement in $\Sigma$ of the vertical segment 
centered at $\tilde x$ with height $2r/5^n$. 
\end{parag}

\begin{parag}[Regularity in space of $\boldsymbol{\rho}$]
\label{s:regsolfin}
The same procedure just described in \S\ref{s:regfin} can be applied to 
the solution $\rho$, which is therefore smooth on the whole $\R^2$ for 
each time $t$.
\end{parag}

\begin{parag}[Regularity in time of $\boldsymbol{u}$ and $\boldsymbol{\rho}$]
\label{r:smoothC}
In \S\ref{s:regfin} and \S\ref{s:regsolfin} we have shown that $u$ and $\rho$
are smooth with respect to space. 
Moreover, they are also smooth in time in each time interval $[T_n,T_{n+1}]$.
This property follows from the regularity with respect to time of the basic 
curves $\Gamma_1(t)$ and $\Gamma_2(t)$.
We observe that $u$ and $\rho$ may fail to be more regular than continuous 
in time at $t=T_n$.
However,  we can apply the procedure described in~\S\ref{ss:timereg} to 
obtain a new velocity field and a new solution that are smooth on 
$\R^2\times [0,T_\infty)$.
\end{parag}

\medskip

What remains to be done is constructing two time-dependent curves 
$\Gamma_1, \Gamma_2$ on the time interval $I=[0,1]$ that satisfy 
Conditions~\ref{p:basiccurvesC.2}, \ref{p:basiccurvesC.3'}, 
\ref{p:basiccurvesC.3''}, \ref{p:basiccurvesC.4'}, and \ref{p:basiccurvesC.5}
in~\S\ref{p:basiccurvesC} and~\S\ref{r:equivalenceC} (and, therefore, 
also all~Conditions \ref{p:basiccurvesC.1}-\ref{p:basiccurvesC.5}). 

Due to the complexity that a rigorous construction would entail, we only 
give a precise description of the initial states $\Gamma_1(0)$, $\Gamma_2(0)$, 
and of the final states~$\Gamma_1(1)$,~$\Gamma_2(1)$ 
(see Figure~\ref{f:gamma-i+f}) and sketch some of the intermediate states 
(see Figures~\ref{f:gammauno}, \ref{f:finesteps}, and~\ref{f:gammadue}). 

\begin{figure}[ht]
\begin{center}
  \includegraphics[scale=1]{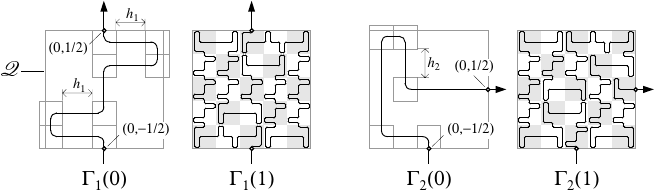}
  \caption{The curves $\Gamma_1$ and $\Gamma_2$ for $t=0$ and $t=1$.}
  \label{f:gamma-i+f}
\end{center}
\end{figure}

\begin{parag}[Construction of $\boldsymbol{\Gamma_1}$
and $\boldsymbol{\Gamma_2}$ for $\boldsymbol{t=0, \, t=1}$]
\label{s:costruzione-g}
We take $\Gamma_1$ and $\Gamma_2$ at the initial time $t=0$
and at the final time $t=1$ as in Figure~\ref{f:gamma-i+f}.
Note that the curves $\Gamma_1(0)$ and $\Gamma_2(0)$
have been obtained by modifying the curves
described in Figure~\ref{f:comb1} 
(which, we recall,  do not satisfy~\ref{p:basiccurvesC.3'} and 
\ref{p:basiccurvesC.4'} in~\S\ref{r:equivalenceC}).
We specify that:
%%%%%
\begin{itemize}
[leftmargin=30pt, itemsep=2pt]

\item
all the small squares drawn in the first and third picture
in Figure~\ref{f:gamma-i+f} have side-length $1/5$;

\item
outside these small squares the curves $\Gamma_1(0)$ and $\Gamma_2(0)$
consists of segments and half-lines;

\item
the intersections of $\Gamma_1(0)$ and $\Gamma_2(0)$
with the small squares agree up to reflections, rotations, 
translations, and scaling by a factor $1/5$ 
with a fixed curve~$\Gamma_0$ in the unit square $\Qone$; the curve $\Gamma_0$ is  
a ``modified quarter of circle'' such as, for example, 
the intersection of the curve $\Gamma_2(0)$ in Figure~\ref{f:comb1} 
with the unit square~$\Qone$;

\item
the parameters $h_1$ and $h_2$ used in the construction are chosen so that
$\Gamma_1(0)$ and $\Gamma_2(0)$ satisfy
~\ref{p:basiccurvesC.3'} and \ref{p:basiccurvesC.4'} in~\S\ref{r:equivalenceC};
in fact, we first choose $h_2$ so that~$\Gamma_2(0)$
satisfies the equal-area requirement of Condition~\ref{p:basiccurvesC.4'}
(we note that $\Gamma_1(0)$ will satisfy this last condition independently
of the choice of $h_1$ for symmetry reasons), and then
we choose $h_1$ so that $\Gamma_1(0)$ and $\Gamma_2(0)$ have
the same length in $\Qone$, that is,  \ref{p:basiccurvesC.3'} holds;%
\footnoteb
{For the consistency of this construction, we need for
the curves $\Gamma_1$ and $\Gamma_2$ to be  contained in 
$\Qone$ at time~$t=0$; that is, we need $h_1,h_2\le 0.3$. 
The curve $\Gamma_0$ mentioned above divides the unit square 
$\Qone$ in two connected components (see~Figure~\ref{f:comb1}). 
We denote by $a$ the area of the smallest one, and by $\ell$ 
the length of $\Gamma_0$. 
We can assume that $\pi/16 \le a \le 1/4$ and 
$\pi/4 \le \ell \le 1$. 
Under these assumptions, one can the check that
$0.225 \le h_1\le 0.264$ and $0.250 \le h_2 \le 0.261$. 
}

\item
finally, both $\Gamma_1(1)$ and $\Gamma_2(1)$ consist of $25$
copies of $\Gamma_1(0)$ and $\Gamma_2(0)$,
rotated, reflected, translated, and scaled by a factor $1/5$;
thus Condition~\ref{p:basiccurvesC.5} is met.
\end{itemize}
\end{parag}

\begin{parag}[Construction of $\boldsymbol{\Gamma_1(t)}$ for
$\boldsymbol{0 < t <1}$]
\label{s:costruzione-g1}
The transition from the curve $\Gamma_1(0)$ to the curve $\Gamma_1(1)$
is given in Figure~\ref{f:gammauno}.
\begin{figure}[h]
\begin{center}
  \includegraphics[scale=1]{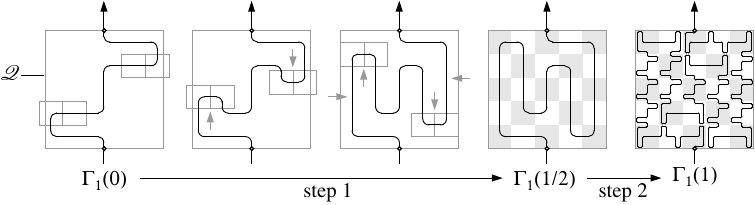}
  \caption{The evolution of $\Gamma_1$ from $t=0$ to $t=1$ in two steps.}
  \label{f:gammauno}
\end{center}
\end{figure}

More precisely, we specify that:
%%%%%
\begin{itemize}
[leftmargin=30pt, itemsep=2pt]

\item
the small squares drawn in the fourth picture have side-length $1/5$, 
and within each of these
squares, $\Gamma_1(1/2)$ agrees with one of the modified quarter
circles used in~\S\ref{s:costruzione-g};

\item
the first step is actually divided in two parts: first, we only move 
the small squares vertically and create the ``kinks'' that appear
in the second square in Figure~\ref{f:gammauno}, then
we keep elongating these kinks with a vertical motion,
but we also add an horizontal motion
(as indicated by the arrows) so that, by time $t=1/2$, 
the small squares are inside the big one;

\item
in the second step, we modify $\Gamma_1(1/2)$
within each square of $\Til_{1/5}$ using one of ``moves'' described
in Figure~\ref{f:finesteps};
\begin{figure}[h]
\begin{center}
  \includegraphics[scale=1]{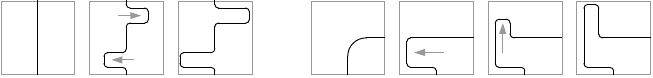}
  \caption{Details of the moves used in the second step
           of the evolution of $\Gamma_1$.}
  \label{f:finesteps}
\end{center}
\end{figure}

\item
it is clear that $\Gamma_1$ satisfies~\ref{p:basiccurvesC.3''}, that is,
$\ell_1(t)$ is increasing in $t$ for $0 \le t \le 1$, 
provided that the horizontal and vertical motions employed in 
the first step are suitably synchronized;

\item
$\Gamma_1$ satisfies the equal-area requirement, Condition ~\ref{p:basiccurvesC.4'},
during the first step because of the symmetry of the evolution;

\item
$\Gamma_1$ satisfies the equal-area requirement, Condition~\ref{p:basiccurvesC.4'},
also during the second step: we remark indeed that
every move of the first type in Figure~\ref{f:finesteps} preserves
the equal-area condition, while the moves of the second type
can be coupled so that to each move that increases the area of
one connected component of~$\Qone\setminus\Gamma_1(t)$
(in some square of $\Til_{1/5}$) corresponds a move that decreases
the area of that component (in another square) by the same amount.
\end{itemize}
\end{parag}

\begin{parag}[Construction of $\boldsymbol{\Gamma_2(t)}$
for $\boldsymbol{0 < t <1}$]
\label{s:costruzione-g2}
The transition from the curve $\Gamma_2(0)$ to the curve $\Gamma_2(1)$
is described in Figure~\ref{f:gammadue}.

\begin{figure}[h]
\begin{center}
  \includegraphics[scale=1]{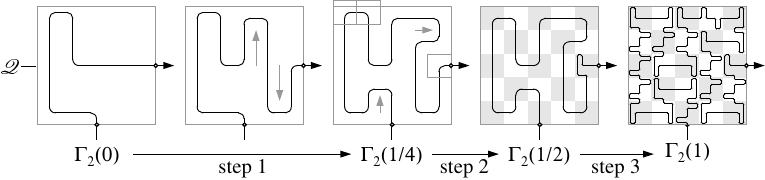}
  \caption{The evolution of $\Gamma_2$ from $t=0$ to $t=1$ in three steps.}
  \label{f:gammadue}
\end{center}
\end{figure}

More precisely, we specify that:
%%%%%
\begin{itemize}
[leftmargin=30pt, itemsep=2pt]

\item
the small squares drawn in the third picture have side-length
$1/5$, and within each of these squares $\Gamma_2(1/4)$ agrees
with one of the ``modified quarter circles'' used
in~\S\ref{s:costruzione-g};

\item
it is clear that ~\ref{p:basiccurvesC.3''} is satisfied
during the entire evolution;

\item
the equal-area requirement~\ref{p:basiccurvesC.4'} is
satisfied during the first step of this evolution
for reasons of symmetry;

\item
most of $\Gamma_2$ is kept fixed by the evolution in the second step,
except the part contained in the two squares in the top-left corner
of $\Qone$, which moves downward, and the part contained in
the middle square on the right side, which evolves according to
the second move in Figure~\ref{f:finesteps};
the equal-area requirement~\ref{p:basiccurvesC.4'} is satisfied at time $t=1/2$,
and is satisfied also for the intermediate times $1/4 < t < 1/2$,
provided that these two moves are suitably synchronized by a change of
time for one of the two basic moves;

\item
in the third step, we modify $\Gamma_2(1/2)$
within each square in $\Til_{1/5}$ (except the middle square on
in the right side of $\Qone$) using one of the moves described
in Figure~\ref{f:finesteps}; note that this step satisfies the
equal-area requirement~\ref{p:basiccurvesC.4'} for the same 
reasons as in  the last step of the evolution of $\Gamma_1$.
\end{itemize}
\end{parag}

%
%	BIBLIOGRAPHY
%
%\bibliography{ACMmixing}

\begin{thebibliography}{10}

\bibitem{ABC1}
G.~Alberti, S.~Bianchini, and G.~Crippa.
\newblock Structure of level sets and {S}ard-type properties of {L}ipschitz
  maps.
\newblock {\em Ann. Sc. Norm. Super. Pisa Cl. Sci. (5)}, 12(4):863--902, 2013.

\bibitem{ABC2}
G.~Alberti, S.~Bianchini, and G.~Crippa.
\newblock A uniqueness result for the continuity equation in two dimensions.
\newblock {\em J.~Eur. Math. Soc. (JEMS)}, 16(2):201--234, 2014.

\bibitem{ACM}
G.~Alberti, G.~Crippa, and A.~L. Mazzucato.
\newblock Exponential self-similar mixing and loss of regularity for continuity
  equations.
\newblock {\em C. R. Math. Acad. Sci. Paris}, 352(11):901--906, 2014.

\bibitem{loss}
G.~Alberti, G.~Crippa, and A.~L. Mazzucato.
\newblock Loss of regularity for continuity equations with non-{L}ipschitz
  velocity.
\newblock 2018.
\newblock Preprint. arXiv:1802.02081.

\bibitem{Amb}
L.~Ambrosio.
\newblock Transport equation and {C}auchy problem for {BV} vector fields.
\newblock {\em Invent. Math.}, 158(2):227--260, 2004.

\bibitem{HW}
L.~Ambrosio and G.~Crippa.
\newblock Continuity equations and {ODE} flows with non-smooth velocity.
\newblock {\em Proc. Roy. Soc. Edinburgh Sect. A}, 144(6):1191--1244, 2014.

\bibitem{Aref}
H.~Aref.
\newblock Stirring by chaotic advection.
\newblock {\em J.~Fluid Mech.}, 143:1--21, 1984.

\bibitem{chemin}
H.~Bahouri, J.-Y. Chemin, and R.~Danchin.
\newblock {\em Fourier analysis and nonlinear partial differential equations}.
\newblock Grundlehren der Mathematischen Wissenschaften, 343. Springer-Verlag,
  Berlin-Heidelberg, 2011.

\bibitem{BMV}
J.~Bedrossian, N.~Masmoudi, and V.~Vicol.
\newblock Enhanced dissipation and inviscid damping in the inviscid limit of
  the {N}avier-{S}tokes equations near the two dimensional {C}ouette flow.
\newblock {\em Arch. Ration. Mech. Anal.}, 219(3):1087--1159, 2016.

\bibitem{bergh}
J.~Bergh and J.~L{\"o}fstr{\"o}m.
\newblock {\em Interpolation spaces. {A}n introduction}.
\newblock Grundlehren der Mathematischen Wissenschaften, 223. Springer-Verlag,
  Berlin-Heidelberg, 1976.

\bibitem{bonicatto}
S.~Bianchini and P.~Bonicatto.
\newblock A uniqueness result for the decomposition of vector fields
  in~${{\mathbf {R}}^d}$.
\newblock 2017.
\newblock Preprint SISSA.

\bibitem{BCCLV00}
G.~Boffetta, A.~Celani, M.~Cencini, G.~Lacorata, and A.~Vulpiani.
\newblock Nonasymptotic properties of transport and mixing.
\newblock {\em Chaos}, 10(1):50--60, 2000.

\bibitem{BC}
F.~Bouchut and G.~Crippa.
\newblock Lagrangian flows for vector fields with gradient given by a singular
  integral.
\newblock {\em J.~Hyperbolic Differ. Equ.}, 10(2):235--282, 2013.

\bibitem{bresch}
D.~Bresch and P.-E. Jabin.
\newblock Global existence of weak solutions for compresssible
  {N}avier--{S}tokes equations: Thermodynamically unstable pressure and
  anisotropic viscous stress tensor.
\newblock {\em Ann. of Math.}, 188:577--684, 2018.

\bibitem{bressan}
A.~Bressan.
\newblock A lemma and a conjecture on the cost of rearrangements.
\newblock {\em Rend. Sem. Mat. Univ. Padova}, 110:97--102, 2003.

\bibitem{sharploss}
E.~Bru\'e and Q.-H. Nguyen.
\newblock Sharp regularity estimates for solutions of the continuity equation
  drifted by {S}obolev vector fields.
\newblock 2018.
\newblock Preprint. arXiv:1806.03466.

\bibitem{ferruccio}
F.~Colombini, T.~Luo, and J.~Rauch.
\newblock Nearly {L}ipschitzean divergence-free transport propagates neither
  continuity nor {BV} regularity.
\newblock {\em Commun. Math. Sci.}, 2(2):207--212, 2004.

\bibitem{CKRZ}
P.~Constantin, A.~Kiselev, L.~Ryzhik, and A.~Zlato{\v{s}}.
\newblock Diffusion and mixing in fluid flow.
\newblock {\em Ann. of Math. (2)}, 168(2):643--674, 2008.

\bibitem{CDL}
G.~Crippa and C.~De~Lellis.
\newblock Estimates and regularity results for the {D}i{P}erna-{L}ions flow.
\newblock {\em J.~Reine Angew. Math.}, 616:15--46, 2008.

\bibitem{schulze}
G.~Crippa and C.~Schulze.
\newblock Cellular mixing with bounded palenstrophy.
\newblock {\em Math. Models Methods Appl. Sci.}, 27(12):2297--2320, 2017.

\bibitem{depauw}
N.~Depauw.
\newblock Non unicit\'e des solutions born\'ees pour un champ de vecteurs {BV}
  en dehors d'un hyperplan.
\newblock {\em C. R. Math. Acad. Sci. Paris}, 337(4):249--252, 2003.

\bibitem{valdinoci}
E.~Di~Nezza, G.~Palatucci, and E.~Valdinoci.
\newblock Hitchhiker's guide to the fractional {S}obolev spaces.
\newblock {\em Bull. Sci. Math.}, 136(5):521--573, 2012.

\bibitem{DPL}
R.~J. DiPerna and P.-L. Lions.
\newblock Ordinary differential equations, transport theory and {S}obolev
  spaces.
\newblock {\em Invent. Math.}, 98(3):511--547, 1989.

\bibitem{FCS14}
D.~P.~G. Foures, C.~P. Caulfield, and P.~J. Schmid.
\newblock Optimal mixing in two-dimensional plane {P}oiseuille flow at finite
  {P}\'eclet number.
\newblock {\em J.~Fluid Mech.}, 748:241--277, 2014.

\bibitem{GW12}
T.~Gotoh and T.~Watanabe.
\newblock Scalar flux in a uniform mean scalar gradient in homogeneous
  isotropic steady turbulence.
\newblock {\em Phys.~D}, 241(3):141--148, 2012.

\bibitem{GDTR09}
E.~Gouillart, O.~Dauchot, J.-L. Thiffeault, and S.~Roux.
\newblock Open-flow mixing: Experimental evidence for strange eigenmodes.
\newblock {\em Phys. Fluids}, 21(2):023603, 2009.

\bibitem{grafakos}
L.~Grafakos.
\newblock {\em Modern {F}ourier analysis}.
\newblock Graduate Texts in Mathematics, 250. Springer-Verlag, New York, third
  edition, 2014.

\bibitem{ikx}
G.~Iyer, A.~Kiselev, and X.~Xu.
\newblock Lower bounds on the mix norm of passive scalars advected by
  incompressible enstrophy-constrained flows.
\newblock {\em Nonlinearity}, 27(5):973--985, 2014.

\bibitem{Jab15}
P.-E. Jabin.
\newblock Critical non-{S}obolev regularity for continuity equations with rough
  velocity fields.
\newblock {\em J.~Differential Equations}, 260(5):4739--4757, 2016.

\bibitem{Jul03}
M.-C. Jullien.
\newblock Dispersion of passive tracers in the direct enstrophy cascade:
  Experimental observations.
\newblock {\em Phys. Fluids}, 15(8):2228--2237, 2003.

\bibitem{JCT00}
M.-C. Jullien, P.~Castiglione, and P.~Tabeling.
\newblock Experimental observation of {B}atchelor dispersion of passive
  tracers.
\newblock {\em Phys. Rev. Lett.}, 85(17):3636--3639, 2000.

\bibitem{KXchemo}
A.~{Kiselev} and X.~{Xu}.
\newblock Suppression of chemotactic explosion by mixing.
\newblock {\em Arch. Ration. Mech. Anal.}, 222(2):1077--1112, 2016.

\bibitem{leger}
F.~L\'eger.
\newblock A new approach to bounds on mixing.
\newblock {\em Math. Models Methods Appl. Sci.}, 2018.
\newblock In press.

\bibitem{doering}
Z.~Lin, J.-L. Thiffeault, and C.~R. Doering.
\newblock Optimal stirring strategies for passive scalar mixing.
\newblock {\em J.~Fluid Mech.}, 675:465--476, 2011.

\bibitem{Liu08}
W.~Liu.
\newblock Mixing enhancement by optimal flow advection.
\newblock {\em SIAM J.~Control Optim.}, 47(2):624--638, 2008.

\bibitem{liverani}
C.~Liverani.
\newblock On contact {A}nosov flows.
\newblock {\em Ann. of Math. (2)}, 159(3):1275--1312, 2004.

\bibitem{llnmd}
E.~Lunasin, Z.~Lin, A.~Novikov, A.~L. Mazzucato, and C.~R. Doering.
\newblock Optimal mixing and optimal stirring for fixed energy, fixed power, or
  fixed palenstrophy flows.
\newblock {\em J.~Math. Phys.}, 53(11):115611, 2012.

\bibitem{MMGVP}
G.~Mathew, I.~Mezi{\'c}, S.~Grivopoulos, U.~Vaidya, and L.~Petzold.
\newblock Optimal control of mixing in {S}tokes fluid flows.
\newblock {\em J.~Fluid Mech.}, 580:261--281, 2007.

\bibitem{MMP}
G.~Mathew, I.~Mezic, and L.~Petzold.
\newblock A multiscale measure for mixing.
\newblock {\em Phys.~D}, 211(1-2):23--46, 2005.

\bibitem{Ottino}
J.~M. Ottino.
\newblock {\em The kinematics of mixing: stretching, chaos, and transport}.
\newblock Cambridge Texts in Applied Mathematics. Cambridge University Press,
  Cambridge, 1989.

\bibitem{RHG99}
D.~Rothstein, E.~Henry, and J.~P. Gollub.
\newblock Persistent patterns in transient chaotic fluid mixing.
\newblock {\em Nature}, 401(6755):770--772, 1999.

\bibitem{seis}
C.~Seis.
\newblock Maximal mixing by incompressible fluid flows.
\newblock {\em Nonlinearity}, 26(12):3279--3289, 2013.

\bibitem{taylor}
M.~Taylor.
\newblock Equivalence of {E}uclidean and toral {S}obolev norms.
\newblock Private communication, 2016.

\bibitem{triebel}
H.~Triebel.
\newblock {\em Theory of function spaces}.
\newblock Monographs in Mathematics, 78. Birkh\"auser Verlag, Basel, 1983.

\bibitem{zlatos}
Y.~Yao and A.~Zlato\v{s}.
\newblock Mixing and un-mixing by incompressible flows.
\newblock {\em J. Eur. Math. Soc. (JEMS)}, 19(7):1911--1948, 2017.

\bibitem{zillinger}
C.~Zillinger.
\newblock On geometric and analytic mixing scales: comparability and
  convergence rates for transport problems.
\newblock 2018.
\newblock Preprint. arXiv:1804.11299.

\end{thebibliography}
\bibliographystyle{plain}

\end{document}